\documentclass[reqno,11pt]{amsart}

\usepackage{amsthm}
\usepackage{amssymb}
\usepackage{amsmath}
\usepackage{amsfonts}

\usepackage{graphicx,psfrag}

\usepackage{subcaption}
\usepackage{comment}

\usepackage[utf8]{inputenc}
\usepackage[english]{babel}

\usepackage[T1]{fontenc}
\usepackage{lmodern}

\usepackage{bookmark}

\usepackage{xcolor}
\usepackage{hyperref}

\hypersetup{
    pdfauthor = {Maycon Sambinelli},
    pdftitle = {Perfect digraphs},
    pdfkeywords = {path partition, diperfect digraphs},
    pdfproducer = {Latex with hyperref},
    backref = page,
    pdfcreator = {pdflatex},
    colorlinks,
    linkcolor = {red!60!black},
    citecolor = {green!60!black},
    urlcolor = {blue!60!black}
}

\usepackage{tikz,nicefrac,pifont,fp}
\usetikzlibrary{arrows,snakes,backgrounds,patterns,matrix,shapes,fit,calc,shadows,plotmarks,arrows.meta,positioning,decorations.pathmorphing,decorations.pathreplacing,decorations.markings,fixedpointarithmetic}

\tikzset{red vertex/.style={circle,draw,minimum size=2mm,inner sep=0pt,outer sep=4pt,fill=red, color=red}}
\tikzset{blue vertex/.style={circle,draw,minimum size=2mm,inner sep=0pt,outer sep=4pt,fill=blue, color=blue}}
\tikzset{black vertex/.style={circle,draw,minimum size=2mm,inner sep=0pt,outer sep=3pt,fill=black, color=black}}
\tikzset{white vertex/.style={circle,draw,minimum size=2mm,inner sep=0pt,outer sep=3pt, color=black}}
\tikzset{green vertex/.style={circle,draw,minimum size=2mm,inner sep=0pt,outer sep=3pt, color=green, fill=green}}

\tikzstyle{color1}=[color=red] 
\tikzstyle{color2}=[color=black]
\tikzstyle{color3}=[color=blue]
\tikzstyle{color4}=[color=green]

\tikzstyle{edge}=[line width=0.9]

\tikzstyle{possible edge}=[edge, dashed]

\tikzstyle{very loosely dotted}=[dash pattern=on \pgflinewidth off 6pt]
\tikzstyle{nonedge}=[color=red, very loosely dotted, line width=1.3]

\tikzstyle{snake}=[decorate, decoration=snake, segment length=1cm]
\tikzstyle{short snake}=[decorate, decoration=snake, segment length=7mm]

\colorlet{setfilling}{green!5!white}
\colorlet{setborder}{gray}

\makeatletter
\let\origsection=\section 
\def\section{\@ifstar{\origsection*}{\mysection}}
\def\mysection{\@startsection{section}{1}\z@{.7\linespacing\@plus\linespacing}{.5\linespacing}{\normalfont\scshape\centering\S}}
\makeatother

\usepackage[abbrev,msc-links,backrefs]{amsrefs}
\usepackage{doi}

\renewcommand{\PrintDOI}[1]{\doi{#1}}

\usepackage{datetime}
\usepackage{lineno}
\newcommand*\patchAmsMathEnvironmentForLineno[1]{
\expandafter\let\csname old#1\expandafter\endcsname\csname #1\endcsname
\expandafter\let\csname oldend#1\expandafter\endcsname\csname end#1\endcsname
\renewenvironment{#1}
{\linenomath\csname old#1\endcsname}
{\csname oldend#1\endcsname\endlinenomath}}
\newcommand*\patchBothAmsMathEnvironmentsForLineno[1]{
\patchAmsMathEnvironmentForLineno{#1}
\patchAmsMathEnvironmentForLineno{#1*}}
\AtBeginDocument{
\patchBothAmsMathEnvironmentsForLineno{equation}
\patchBothAmsMathEnvironmentsForLineno{align}
\patchBothAmsMathEnvironmentsForLineno{flalign}
\patchBothAmsMathEnvironmentsForLineno{alignat}
\patchBothAmsMathEnvironmentsForLineno{gather}
\patchBothAmsMathEnvironmentsForLineno{multline}
}

\usepackage{geometry}
\newtheorem{theorem}[equation]{Theorem}
\newtheorem{lemma}[equation]{Lemma}
\newtheorem{proposition}[equation]{Proposition}
\newtheorem{conjecture}[equation]{Conjecture}
\newtheorem{corollary}[equation]{Corollary}

\theoremstyle{definition}

\theoremstyle{case}

\numberwithin{equation}{section}

\newcommand{\fD}{\mathfrak{D}}
\newcommand{\fB}{\mathfrak{B}}

\newcommand{\newmathsymb}[2]{}

\newcommand{\SBE}{\(S_\text{BE}\)}

\usepackage{xspace}

\newcommand{\Set}[1]{\{#1\}}

\newcommand{\sC}{\mathcal{C}}

\newcommand{\sP}{\mathcal{P}}
\newcommand{\sQ}{\mathcal{Q}}

\newcommand\tand{\ \text{and}\ }
\let\:\colon

\begin{document}

\title{Perfect digraphs}

\author{C. N. da Silva \and O. Lee \and M. Sambinelli}

\shortdate
\yyyymmdddate
\settimeformat{ampmtime}
\date{\today, \currenttime}

\address{Departamento de Computação, Universidade Federal de São Carlos, Sorocaba, Brazil}
\email{candidal@ufscar.br}

\address{Instituto de Computação, Universidade Estadual de Campinas, Campinas, Brazil}
\email{lee@ic.unicamp.br}

\address{Instituto de Matemática e Estatística, Universidade de São Paulo, São Paulo, Brazil}
\email{sambinelli@ime.usp.br}

\thanks{
O. Lee was partially support by FAPESP (grant 2015/11937-9) and CNPq (grants 311373/2015-1 and 425340/2016-3).
M. Sambinelli was partially supported by National Counsel of Technological and Scientific Development (CNPq) (Grant 141216/2016-6) and São Paulo Research Foundation (FAPESP) (Grant~2017/23623-4).
This study was financed in part by the Coordenação de Aperfeiçoamento de Pessoal de Nível Superior, Brasil (CAPES), Finance Code 001.}

\begin{abstract}
  Let \(D\) be a digraph.
  Given a set of vertices \(S \subseteq V(D)\), an \emph{\(S\)-path partition} \(\sP\) of \(D\) is a collection of paths of \(D\) such that \(\{V(P) \: P \in \sP\}\) is a partition of \(V(D)\) and \(|V(P) \cap S| = 1\) for every \(P \in \sP\).
  We say that \(D\) satisfies the \emph{\(\alpha\)-property} if, for every maximum stable set \(S\) of \(D\), there exists an \(S\)-path partition of \(D\), and we say that \(D\) is \emph{\(\alpha\)-diperfect} if every induced subdigraph of \(D\) satisfies the \(\alpha\)-property.
  A digraph \(C\) is an \emph{anti-directed odd cycle} if (i) the underlying graph of \(C\) is a cycle \(x_1x_2 \cdots x_{2k + 1}x_1\), where \(k \in \mathbb{Z}\) and \(k \geq 2\), and (ii) each of the vertices \(x_1, x_2, x_3, x_4, x_6,\) \(x_8, \ldots, x_{2k}\) is either a source or a sink.
  Berge (1982) conjectured that a digraph is \(\alpha\)-diperfect if, and only if, it contains no induced anti-directed odd cycle.
  Remark that this conjecture is strikingly similar to Berge's conjecture on perfect graphs -- nowadays known as the Strong Perfect Graph Theorem (Chudnovsky, Robertson, Seymour, and Thomas, 2006).
  To the best of our knowledge, Berge's conjecture for \(\alpha\)-diperfect digraphs has been verified only for symmetric digraphs and digraphs whose underlying graph are perfect.
  In this paper, we verify it for digraphs whose underlying graphs are series-parallel and for in-semicomplete digraphs.
  Moreover, we propose a conjecture similar to Berge's and verify it for all the known cases of Berge's conjecture.
\end{abstract}

\maketitle

\section{Introduction}
\label{sec:introduction}

All digraphs considered here do not contain loops or parallel arcs (but may contain cycles of length two). 
Terminology and notation used are standard and we refer the reader to Bondy and Murty's Book~\cite{BondyMurty2008} for missing definitions.
Given a digraph \(D\), we denote its vertex set by~\(V(D)\) and its arc set by~\(A(D)\). 
Given a pair of vertices \(u, v \in V(D)\), we write \(uv\) to denote an arc
with tail \(u\) and head \(v\), and we say that \(u\) and \(v\) are \emph{adjacent} in \(D\) if \(\{uv, vu\} \cap A(D) \neq \emptyset\); otherwise we say that \(u\) and \(v\) are \emph{non-adjacent}.
A \emph{stable set} of \(D\) is a set \(S \subseteq V(D)\) such that every pair of vertices in \(S\) are non-adjacent in \(D\).
We say that a stable set~\(S\) of \(D\) is \emph{maximum} if for every stable set \(S'\) of \(D\) it follows that \(|S'| \leq |S|\).
The \emph{stability number} of \(D\), denoted by \(\alpha(D)\), is the size of a maximum stable set of~\(D\).
A \emph{path}~\(P\) of a digraph~\(D\) is a sequence \(v_1v_2 \cdots v_\ell\) of distinct vertices of~\(D\) such that \(v_iv_{i + 1} \in A(D)\), for \(i = 1, 2, \ldots, \ell - 1\).
Sometimes, when convenient, we treat \(P\) as being the subdigraph of \(D\) with vertex set \(\{v_0, v_1, \ldots, v_\ell\}\) and arc set \(\{v_iv_{i + 1} \colon i = 0, 1, \ldots, \ell - 1\}\).
A collection of paths \(\sP = \{P_1, P_2, \ldots, P_\ell\}\) is a \emph{path partition} of~\(D\) if \(\{V(P_1), V(P_2), \ldots, V(P_\ell)\}\) is a partition of \(V(D)\).
We denote by \(\pi(D)\) the size of a smallest path partition of \(D\).
The following is a classical result by Gallai and Milgram~\cite{GallaiMilgram1960}.

\begin{theorem}[Gallai and Milgram, 1960]\label{the:gallai-milgram}
	For every digraph \(D\), we have \(\pi(D) \leq \alpha(D)\).
\end{theorem}

Given a digraph \(D\) and a stable set \(S \subseteq V(D)\), an \emph{\(S\)-path partition} \(\sP\) of \(D\) is a path partition where each path  in \(\sP\) has precisely one vertex in \(S\), i.e., \(|V(P) \cap S| = 1\) for all \(P \in \sP\). 
In the early 80's, Berge~\cite{Berge1982b} noticed that although several proofs of Theorem~\ref{the:gallai-milgram} had been discovered until then, none implied that every digraph~\(D\) contains an \(S\)-path partition \(\sP\) for some maximum stable set \(S\).
Note that the existence of such \(S\)-path partition would imply in Theorem~\ref{the:gallai-milgram}, since \(\pi(D) \leq |\sP| = |S| = \alpha(D)\) (later Meyniel~\cite{Meyniel1986} showed that there are digraphs which admits no \(S\)-path partition for every maximum stable set \(S\)).
This led Berge~\cite{Berge1982b} to  propose the class of \(\alpha\)-diperfect digraphs.
We say that a digraph \(D\) satisfies the \emph{\(\alpha\)-property} if, for every maximum stable set \(S\) of \(D\), there exists an \(S\)-path partition of \(D\), and we say that \(D\) is \emph{\(\alpha\)-diperfect} if every induced subdigraph of~\(D\) satisfies the \(\alpha\)-property.

Given a digraph \(D\), we denote its underlying graph by \(U(D)\) (in this text we always consider that the underlying graph is simple).
A digraph \(C\) is an \emph{anti-directed odd cycle} if (i) \(U(C) = x_1x_2 \cdots x_{2k + 1}x_1\) is a cycle, where \(k \in \mathbb{Z}\) and \(k \geq 2\), and (ii) each of the vertices \(x_1, x_2, x_3, x_4, x_6,\) \(x_8, \ldots, x_{2k}\) is either a source or a sink in~\(D\) (see Figure~\ref{fig:anti-directed}).
Berge~\cite{Berge1982b} showed that anti-directed odd cycles do not satisfy the \(\alpha\)-property, and hence are not \(\alpha\)-diperfect, which led him to conjecture the following characterization for \(\alpha\)-diperfect digraphs~\cite{Berge1982b}.
Note that it is strikingly similar to Berge's conjecture on perfect graphs -- nowadays known as the Strong Perfect Graph Theorem~\cite{ChudnovskyEtAl2006} (see Theorem~\ref{the:strong-perfect-graphs}).

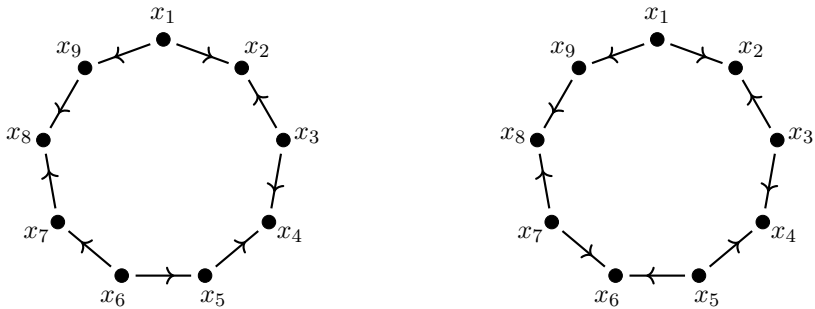
\begin{figure}
	\begin{subfigure}{0.4\textwidth}
		\centering\scalebox{0.9}{\tikzset{middlearrow/.style={
	decoration={markings,
		mark= at position 0.7 with {\arrow{#1}},
	},
	postaction={decorate}
}}

\tikzset{digon/.style={
	decoration={markings,
		mark= at position 0.4 with {\arrow[inner sep=10pt]{<}},
		mark= at position 0.8 with {\arrow[inner sep=10pt]{>}},
	},
	postaction={decorate}
}}

\begin{tikzpicture}[scale = 0.6]

	\def\myr{3}

	\node (x6) [black vertex] at (0*40-110:\myr) {};
	\node (label_x6) [] at      (0*40-110:\myr + 0.6) {\(x_6\)};

	\node (x5) [black vertex] at (1*40-110:\myr) {};
	\node (label_x5) [] at      (1*40-110:\myr + 0.6) {\(x_5\)};

	\node (x4) [black vertex] at (2*40-110:\myr) {};
	\node (label_x4) [] at      (2*40-110:\myr + 0.6) {\(x_4\)};

	\node (x3) [black vertex] at (3*40-110:\myr) {};
	\node (label_x3) [] at      (3*40-110:\myr + 0.6) {\(x_3\)};

	\node (x2) [black vertex] at (4*40-110:\myr) {};
	\node (label_x2) [] at      (4*40-110:\myr + 0.6) {\(x_2\)};

	\node (x1) [black vertex] at (5*40-110:\myr) {};
	\node (label_x1) [] at      (5*40-110:\myr + 0.6) {\(x_1\)};

	\node (x9) [black vertex] at (6*40-110:\myr) {};
	\node (label_x9) [] at      (6*40-110:\myr + 0.6) {\(x_9\)};

	\node (x8) [black vertex] at (7*40-110:\myr) {};
	\node (label_x8) [] at      (7*40-110:\myr + 0.6) {\(x_8\)};

	\node (x7) [black vertex] at (8*40-110:\myr) {};
	\node (label_x7) [] at      (8*40-110:\myr + 0.6) {\(x_7\)};

	\draw[edge, middlearrow={>}] (x1) -- (x2);
	\draw[edge, middlearrow={>}] (x3) -- (x2);
	\draw[edge, middlearrow={>}] (x3) -- (x4);
	\draw[edge, middlearrow={>}] (x5) -- (x4);
	\draw[edge, middlearrow={>}] (x6) -- (x5);
	\draw[edge, middlearrow={>}] (x6) -- (x7);
	\draw[edge, middlearrow={>}] (x7) -- (x8);
	\draw[edge, middlearrow={>}] (x9) -- (x8);
	\draw[edge, middlearrow={>}] (x1) -- (x9);

\end{tikzpicture}

 }
	\end{subfigure}
	\begin{subfigure}{0.4\textwidth}
		\centering\scalebox{0.9}{\tikzset{middlearrow/.style={
	decoration={markings,
		mark= at position 0.7 with {\arrow{#1}},
	},
	postaction={decorate}
}}

\tikzset{digon/.style={
	decoration={markings,
		mark= at position 0.4 with {\arrow[inner sep=10pt]{<}},
		mark= at position 0.8 with {\arrow[inner sep=10pt]{>}},
	},
	postaction={decorate}
}}

\begin{tikzpicture}[scale = 0.6]

	\def\myr{3}

	\node (x6) [black vertex] at (0*40-110:\myr) {};
	\node (label_x6) [] at      (0*40-110:\myr + 0.6) {\(x_6\)};

	\node (x5) [black vertex] at (1*40-110:\myr) {};
	\node (label_x5) [] at      (1*40-110:\myr + 0.6) {\(x_5\)};

	\node (x4) [black vertex] at (2*40-110:\myr) {};
	\node (label_x4) [] at      (2*40-110:\myr + 0.6) {\(x_4\)};

	\node (x3) [black vertex] at (3*40-110:\myr) {};
	\node (label_x3) [] at      (3*40-110:\myr + 0.6) {\(x_3\)};

	\node (x2) [black vertex] at (4*40-110:\myr) {};
	\node (label_x2) [] at      (4*40-110:\myr + 0.6) {\(x_2\)};

	\node (x1) [black vertex] at (5*40-110:\myr) {};
	\node (label_x1) [] at      (5*40-110:\myr + 0.6) {\(x_1\)};

	\node (x9) [black vertex] at (6*40-110:\myr) {};
	\node (label_x9) [] at      (6*40-110:\myr + 0.6) {\(x_9\)};

	\node (x8) [black vertex] at (7*40-110:\myr) {};
	\node (label_x8) [] at      (7*40-110:\myr + 0.6) {\(x_8\)};

	\node (x7) [black vertex] at (8*40-110:\myr) {};
	\node (label_x7) [] at      (8*40-110:\myr + 0.6) {\(x_7\)};

	\draw[edge, middlearrow={>}] (x1) -- (x2);
	\draw[edge, middlearrow={>}] (x3) -- (x2);
	\draw[edge, middlearrow={>}] (x3) -- (x4);
	\draw[edge, middlearrow={>}] (x5) -- (x4);
	\draw[edge, middlearrow={>}] (x5) -- (x6);
	\draw[edge, middlearrow={>}] (x7) -- (x6);
	\draw[edge, middlearrow={>}] (x7) -- (x8);
	\draw[edge, middlearrow={>}] (x9) -- (x8);
	\draw[edge, middlearrow={>}] (x1) -- (x9);

\end{tikzpicture}

 }
	\end{subfigure}
	\caption{All possible non-isomorphic anti-directed odd cycles of order \(9\).}
    \label{fig:anti-directed}
\end{figure}

\begin{conjecture}[Berge, 1982]\label{con:diperfect}
	A digraph \(D\) is \(\alpha\)-diperfect if, and only if, \(D\) contains no induced anti-directed odd cycle.
\end{conjecture}

\begin{theorem}[Chudnovsky, Robertson, Seymour, and Thomas, 2006]\label{the:strong-perfect-graphs}
	A graph \(G\) is perfect if, and only if, neither \(G\) nor its complement contains an induced odd cycle of order at least \(5\).
\end{theorem}
 
The necessity of Conjecture~\ref{con:diperfect} was verified by Berge~\cite{Berge1982b}, since he proved that anti-directed odd cycles do not satisfy the \(\alpha\)-property, but the sufficiency remains open.
Berge~\cite{Berge1982b} also verified Conjecture~\ref{con:diperfect} for digraphs whose underlying graphs are perfect and for symmetric digraphs.
To the best of our knowledge these are the only particular cases verified for this conjecture so far.

A graph \(G\) is \emph{series-parallel} if it can be obtained from the null graph by applying the following operations repeatedly: (i)~adding a vertex \(v\) with degree at most one; (ii)~adding a loop; (iii)~adding a parallel edge; (iv)~subdividing an edge.
A well-known characterization of series-parallel graphs is that a graph is series-parallel if, and only if, it contains no subdivision of~\(K_4\) (complete graph with order~\(4\)).
Series-parallel graphs are a classical class of graphs and a common start point towards verifying graph theoretical conjectures~\cite{EhrenmullerFernandesHeise17,JuvanMoharThomas99,Merker15,NoRo14,LeWa01}.
A digraph \(D\) is \emph{(locally) in-semicomplete} if, for every vertex \(v \in V(D)\), the in-neighborhood  of \(v\) induces a semicomplete digraph.
Note that out-trees, directed cycles, and semicomplete digraphs are all in-semicomplete digraphs.
This class was introduced by Bang-Jensen~\cite{Bang-Jensen1990,Bang-Jensen1995} and have been well studied in the literature~\cite{GuoVolkmann1994,Bang-JensenEtAl1997,GuoVolkmann1994b,Galeana-SanchezOlsen2016,Huang1998}.
For instance, Bondy's Conjecture and Laborde-Payan-Xuong's Conjecture remain open for general digraphs but were verified for this class~\cite{Bang-JensenEtAl2006,Galeana-SanchezGomes2008}.
We refer the reader to the book by Bang-Jensen and Gutin~\cite{Bang-JensenGutin2008} for further information on in-semicomplete digraphs and related classes. 

In this work, we verify Conjecture~\ref{con:diperfect} for digraphs whose underlying graphs are series-parallel (Section~\ref{sec:series-parallel}), in-semicomplete digraphs (Section~\ref{sec:in-semicomplete}), and for a small extension of symmetric digraphs (Section~\ref{sec:k-semi-symmetric}).

The lack of results for Conjecture~\ref{con:diperfect} and the complexity of the proof of Theorem~\ref{the:strong-perfect-graphs}, led us to believe that Conjecture~\ref{con:diperfect} is a very challenging problem.
Trying to understand this difficulty and hoping to obtain a deeper insight on this problem, we decided to study a class of digraphs that we named \emph{BE-diperfect}.
Given a digraph \(D\) and a stable set \(S \subseteq V(D)\), a path partition \(\sP\) is an \emph{\(S_{\rm BE}\)-path partition} of \(D\) if for every \(P = x_1x_2 \cdots x_k \in \sP\), we have \(V(P) \cap S = \{x\}\), where \(x \in \{x_1, x_k\}\).
A digraph \(D\) satisfies the \emph{BE-property} (short for \emph{Begin-End property}) if there exists an \(S_{\rm BE}\)-path partition for every maximum stable set \(S\) of \(D\), and we say that \(D\) is \emph{BE-diperfect} if every induced subdigraph of \(D\) satisfies the BE-property.
Clearly if a digraph \(D\) is BE-diperfect then it is also \(\alpha\)-diperfect.

\begin{figure}
	\begin{subfigure}{0.4\textwidth}
		\centering\scalebox{0.9}{\tikzset{middlearrow/.style={
	decoration={markings,
		mark= at position 0.6 with {\arrow{#1}},
	},
	postaction={decorate}
}}

\tikzset{digon/.style={
	decoration={markings,
		mark= at position 0.4 with {\arrow[inner sep=10pt]{<}},
		mark= at position 0.8 with {\arrow[inner sep=10pt]{>}},
	},
	postaction={decorate}
}}

\begin{tikzpicture}[scale = 0.6]

	\def\myr{2}

	\node (x3) [black vertex] at (0*120+90:\myr) {};
	\node (label_x3) [] at      (0*120+90:\myr + 0.6) {\(x_3\)};

	\node (x2) [black vertex] at (2*120+90:\myr) {};
	\node (label_x2) [] at      (2*120+90:\myr + 0.6) {\(x_2\)};

	\node (x1) [black vertex] at (1*120+90:\myr) {};
	\node (label_x1) [] at      (1*120+90:\myr + 0.6) {\(x_1\)};

	\draw[edge, middlearrow={>}] (x1) -- (x2);
	\draw[edge, middlearrow={>}] (x1) -- (x3);
	\draw[edge, middlearrow={>}] (x3) -- (x2);

\end{tikzpicture}

 }
	\end{subfigure}
	\begin{subfigure}{0.4\textwidth}
		\centering\scalebox{0.9}{\tikzset{middlearrow/.style={
	decoration={markings,
		mark= at position 0.6 with {\arrow{#1}},
	},
	postaction={decorate}
}}

\tikzset{digon/.style={
	decoration={markings,
		mark= at position 0.4 with {\arrow[inner sep=10pt]{<}},
		mark= at position 0.8 with {\arrow[inner sep=10pt]{>}},
	},
	postaction={decorate}
}}

\begin{tikzpicture}[scale = 0.6]

	\def\myr{2}
	\def\rot{193}

	\foreach \i in {1, 2, 3, 4, 5, 6, 7} 
	{
		\node (x\i) [black vertex] at (\i*51.4+\rot:\myr) {};
		\node (label_x\i) [] at      (\i*51.4+\rot:\myr + 0.6) {\(x_\i\)};
	}

	\draw[edge, middlearrow={>}] (x1) -- (x2);
	\draw[edge, middlearrow={>}] (x1) -- (x7);
	\draw[edge, middlearrow={>}] (x3) -- (x2);
	\draw[edge, middlearrow={>}] (x4) -- (x3);
	\draw[edge, middlearrow={>}] (x5) to [bend left] (x4);
	\draw[edge, middlearrow={>}] (x4) to [bend left] (x5);
	\draw[edge, middlearrow={>}] (x6) to [bend left] (x7);
	\draw[edge, middlearrow={>}] (x7) to [bend left] (x6);
	\draw[edge, middlearrow={>}] (x5) to  (x6);

\end{tikzpicture}

 }
	\end{subfigure}
	\caption{Examples of blocking odd cycles.}\label{fig:blocking-odd-cycle}
\end{figure}
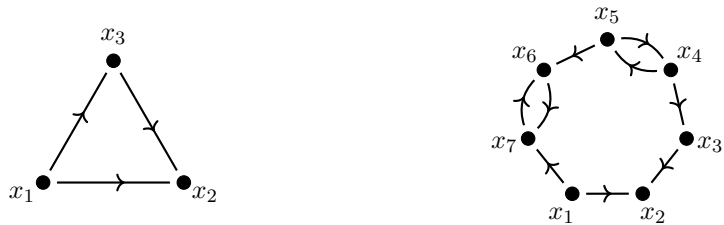

A digraph \(C\) is a \emph{blocking odd cycle} if \(U(C) = x_1x_2 \cdots x_{2k + 1}x_1\) is a cycle, where \(k\) is a positive integer, and each \(x_1\)  and \(x_2\) is either a source or a sink.
For examples of blocking odd cycles see Figure~\ref{fig:blocking-odd-cycle} -- the digraph on the left is called \emph{transitive triangle}.
In this case, we say that \((x_1, x_2)\) is a \emph{blocking pair} of \(C\).
Note that every anti-directed odd cycle is also a blocking odd cycle.
The following proposition shows that blocking odd cycles do not satisfy the BE-property.

\begin{proposition}\label{prop:bed-blocking-cycle}
	If \(D\) is a blocking odd cycle, then \(D\) does not satisfy the BE-property.
\end{proposition}
\begin{proof}
  Let \(D\) be a blocking odd cycle with \(U(D) = x_1x_2 \cdots x_{2k+ 1} x_1\).
	Suppose, without loss of generality, that \((x_1, x_2)\) is a blocking pair and  \(x_1\) is a sink.
	Let \(S = \{x_3,x_5,\ldots,x_{2k+1}\}\), and note that \(S\) is a maximum stable set of \(D\).

	We claim that there is no \(S_\text{BE}\)-path partition in \(D\).
	If \(k=1\), then \(D\) is a transitive triangle and \(S=\{x_3\}\), and the claim follows trivially.
	Thus suppose that \(k\geq 2\), and hence \(x_3 \neq x_{2k + 1}\).
	Towards a contradiction, suppose that there exists an \SBE-path partition \(\sP\) of \(D\).
 	Thus \(x_2x_3\) is a path in~\(\sP\), since this is the only path in \(D\) which contains \(x_2\) and has an end in \(S\) (see Figure~\ref{fig:blocking-odd-cycle-not-bed}).
 	By the same argument, \(x_{2k + 1}x_1\) is a path in \(\sP\).
	Since \(x_2x_3 \in \sP\) and \(\sP\) is an \SBE-path partition,  \(x_4x_5\) or \(x_5x_4\) is a path in \(\sP\), which implies that  \(x_6x_7\) or \(x_7x_6\) is a path in \(\sP\), and so on.
	In particular, this means that  \(x_{2k}x_{2k + 1}\) or \(x_{2k + 1}x_{2k}\) is a path in \(\sP\), a contradiction since \(x_{2k + 1}x_1 \in \sP\) and \(\sP\) is vertex-disjoint.
	
	Thus there is no \SBE-path partition in \(D\), and hence \(D\) does not satisfy the BE-property.
\end{proof}

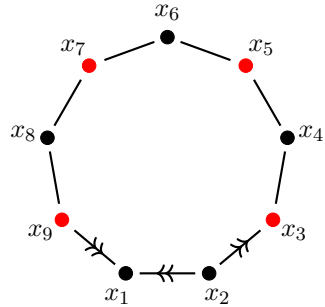
\begin{figure}[h]
    \centering\scalebox{0.9}{\tikzset{middlearrow/.style={
	decoration={markings,
		mark= at position 0.7 with {\arrow{#1}},
	},
	postaction={decorate}
}}

\tikzset{digon/.style={
	decoration={markings,
		mark= at position 0.4 with {\arrow[inner sep=10pt]{<}},
		mark= at position 0.8 with {\arrow[inner sep=10pt]{>}},
	},
	postaction={decorate}
}}

\begin{tikzpicture}[scale = 0.6]

	\def\myr{3}

	\node (x1) [black vertex] at (0*40-110:\myr) {};
	\node (label_x1) [] at      (0*40-110:\myr + 0.6) {\(x_1\)};

	\node (x2) [black vertex] at (1*40-110:\myr) {};
	\node (label_x1) [] at      (1*40-110:\myr + 0.6) {\(x_2\)};

	\node (x3) [red vertex] at (2*40-110:\myr) {};
	\node (label_x1) [] at      (2*40-110:\myr + 0.6) {\(x_3\)};

	\node (x4) [black vertex] at (3*40-110:\myr) {};
	\node (label_x1) [] at      (3*40-110:\myr + 0.6) {\(x_4\)};

	\node (x5) [red vertex] at (4*40-110:\myr) {};
	\node (label_x1) [] at      (4*40-110:\myr + 0.6) {\(x_5\)};

	\node (x6) [black vertex] at (5*40-110:\myr) {};
	\node (label_x1) [] at      (5*40-110:\myr + 0.6) {\(x_6\)};

	\node (x7) [red vertex] at (6*40-110:\myr) {};
	\node (label_x1) [] at      (6*40-110:\myr + 0.6) {\(x_7\)};

	\node (x8) [black vertex] at (7*40-110:\myr) {};
	\node (label_x1) [] at      (7*40-110:\myr + 0.6) {\(x_8\)};

	\node (x9) [red vertex] at (8*40-110:\myr) {};
	\node (label_x1) [] at      (8*40-110:\myr + 0.6) {\(x_9\)};

	\draw[edge, middlearrow={>>}] (x2) -- (x1);
	\draw[edge, middlearrow={>>}] (x2) -- (x3);
	\draw[edge, middlearrow={>>}] (x9) -- (x1);
	\draw[edge] (x3) -- (x4) -- (x5) -- (x6) -- (x7) -- (x8) -- (x9);
\end{tikzpicture}

 }
    \caption[Example of a blocking odd cycle]{Example of a blocking odd cycle with \(k = 4\).
    The vertices in \(S\) are shown in red. We use an arrow with two heads from a vertex \(u\) to a vertex \(v\) to denote that the arc \(uv\) is present and the arc \(vu\) is missing in the digraph. Moreover, we join two vertices by a line (without heads) to denote that they are adjacent.}
    \label{fig:blocking-odd-cycle-not-bed}
\end{figure}

Proposition~\ref{prop:bed-blocking-cycle} led us to propose the following characterization of BE-diperfect digraphs, very similar to the one proposed by Berge to \(\alpha\)-diperfect digraphs.

\begin{conjecture}[Begin-End conjecture]\label{con:bed}
	A digraph \(D\) is BE-diperfect if, and only if, \(D\) has no blocking odd cycle as an induced subdigraph.
\end{conjecture}

Let us compare Conjectures~\ref{con:diperfect} and~\ref{con:bed}.
As mentioned before, every BE-diperfect digraph is an \(\alpha\)-diperfect digraph.
Moreover, it is not hard to see that the set of all BE-diperfect digraphs is properly contained in the set of all \(\alpha\)-diperfect digraphs; take for example, a transitive triangle which is \(\alpha\)-diperfect but is not \(BE\)-diperfect.
Let \(\fB\) be the set of all digraphs without an induced anti-directed odd cycle and let \(\fD\) be the set of all digraphs without an induced blocking odd cycle.
Since every anti-directed odd cycle is a blocking odd cycle, it follows that \(\fD \subset \fB\), as expected, since we can restate Conjecture~\ref{con:diperfect} as \emph{``A digraph \(D\) is \(\alpha\)-diperfect if, and only if, \(D \in \fB\)''} and Conjecture~\ref{con:bed} as \emph{``A digraph \(D\) is BE-diperfect if, and only if, \(D \in \fD\)''}. 
Note that in Conjecture~\ref{con:bed}, we aim to characterize a smaller subclass of \(\alpha\)-diperfect digraphs, but we require a larger class of forbidden subdigraphs.
So neither conjecture implies the other.

In addition to our results for Conjecture~\ref{con:diperfect}, we provide evidence to support Conjecture~\ref{con:bed} by verifying it for 
 digraphs whose underlying graphs are perfect or series-parallel, in-semicomplete digraphs, and for a small extension of symmetric digraphs.
We remark that these are all the known cases for which Conjecture~\ref{con:diperfect} has been verified for.

The remaining of this text is organized in the following way.
In Section~\ref{sec:definitions}, we present some basic definitions and notations.
In Section~\ref{sec:perfect}, we verify Conjecture~\ref{con:bed} for digraphs whose underlying graphs are perfect, since we need this results for Section~\ref{sec:structural-results}, where we present some structural results regarding Conjectures~\ref{con:diperfect} and~\ref{con:bed}.
In Section~\ref{sec:partition-lemmas}, we present two auxiliary results that allow us to split a digraph \(D\) into two proper induced subdigraphs  \(H_1\) and \(H_2\) such that \(\alpha(H) = \alpha(H_1) + \alpha(H_2)\).
In latter sections, we are going to use this fact in the induction step of our proofs.
We verify Conjectures~\ref{con:diperfect} and~\ref{con:bed} for digraphs whose underlying graphs are series-parallel in Section~\ref{sec:series-parallel}, for in-semicomplete digraphs in Section~\ref{sec:in-semicomplete}, and for a small extension of symmetric digraphs in Section~\ref{sec:k-semi-symmetric}.
Finally, in Section~\ref{sec:concluding-remarks}, we present some concluding remarks.

\section{Basic definitions}
\label{sec:definitions}

In this section, we present some basic definitions and notation.
Given a digraph (resp.\ graph)~\(G\), we denote its vertex set by \(V(G)\) and its arc set (resp.\ edge set) by~\(A(G)\) (resp.~\(E(G)\)).

Let \(G\) be a (di)graph.
A set of vertices \(W\) of \(G\) is a \emph{clique} if for every pair of vertices \(u, v \in W\) the vertices \(u\) and \(v\) are adjacent.
Given a set \(X \subseteq V(G)\), we write \(G[X]\) to denote the sub(di)graph of \(G\) induced by \(X\) and \(G - X\) to denote the (di)graph \(G[V(G) \setminus X]\).
If \(F\) is a set of edges (resp.\ arcs), we write \(G - F\) to denote the graph (resp.\ digraph) with vertex set \(V(G)\) and edge (resp.\ arc) set \(E(G) \setminus F\) (resp.\ \(A(G) \setminus F\)).
We write \(H \subseteq G\) to denote that \(H\) is a sub(di)graph of the (di)graph \(G\).

A \emph{Hamilton path} \(P\) of a (di)graph \(G\) is a path such that \(V(P) = V(G)\), and a \emph{Hamilton cycle} \(C\) of \(G\) is a cycle such that \(V(C) = V(G)\).
If \(G\) is a digraph, by path and cycle we always means directed path and directed cycle, respectively.
We say that \(G\) is \emph{hamiltonian} if it contains a Hamilton cycle.

The set of neighbors of a vertex \(v\) of a graph \(G\) is denoted by \(N_G(v)\), or simply by \(N(v)\) (here, as elsewhere, we drop the index referring to the (di)graph when this is clear from the context).
Given two vertices \(u\) and \(v\) of a digraph \(D\), we say that \(u\) is an \emph{in-neighbor} of \(v\)  if \(uv \in A(D)\), and that it is an \emph{out-neighbor} of \(v\) if \(vu \in A(D)\).
The set of \emph{in-neighbors} and \emph{out-neighbors} of a vertex \(v\) of a digraph \(D\) are denoted by \(N^-_D(v)\) and \(N^+_D(v)\), respectively.  
A vertex \(v\) is a \emph{source}  if \(N^-(v) = \emptyset\) and a \emph{sink} if \(N^+(v) = \emptyset\).

Let \(D\) be a digraph and let \(X \subset V(D)\) and \(Y \subset V(D)\) be two non-empty disjoint sets of vertices.
We write \(X \mapsto_D Y\) (resp.\ \(X \leftrightarrow_D Y\)) to denote that, for every pair of vertices \(x \in X\) and \(y \in Y\), we have \(xy \in A(D)\) and \(yx \notin A(D)\) (resp.\ \(\{xy, yx\} \subseteq A(D)\)).
When \(X = \{x\}\) or \(Y = \{y\}\) (or both), we simply write the element, as in the following examples: \(x \mapsto_D Y\), \(X \mapsto_D y\), \(x \leftrightarrow_D Y\), \(x \leftrightarrow_D y\).

A graph \(G\) is  \emph{null} if \(V(G) = E(G) = \emptyset\).
A digraph \(D\) is \emph{strong} if there exists a path from \(u\) to \(v\) for every pair of vertices \(u, v \in V(D)\).
A \emph{strong component} of \(D\) is a maximal induced subdigraph of \(D\) which is strong.
A strong component \(X\) of \(D\) is \emph{minimal} if no arc enters \(X\), that is, there exists no arc \(yx\) such that \(x \in V(X)\) and \(y \in V(D) \setminus V(X)\).

A digraph \(D\) is a \emph{tournament} if for every pair of vertices \(u, v \in V(D)\) we have either \(u \mapsto v\) or \(v \mapsto u\).
We say that \(D\) is \emph{semicomplete} if every pair of its vertices are adjacent, and we say that \(D\) is \emph{complete} if for every pair \(u, v \in V(D)\), we have \(u \leftrightarrow v\).
A digraph \(D\) is \emph{symmetric} if for every pair of adjacent vertices \(u, v \in V(D)\), we have \(u \leftrightarrow v\).
The \emph{inverse digraph} of \(D\) is the digraph with vertex set \(V(D)\) and arc set \(\{uv \: vu \in A(D)\}\).

The \emph{underlying graph} of a digraph \(D\), denoted by \(U(D)\), is the simple graph defined by \(V(U(D)) = V(D)\) and \(E(U(D)) = \{uv \: u \tand v \text{ are adjacent in } D\}\).

\section{Digraphs whose underlying graphs are perfect}
\label{sec:perfect}

In this section, we verify Conjecture~\ref{con:bed} for digraphs whose underlying graphs are perfect.
A graph~\(G\) is \emph{perfect} if, for every induced subgraph \(H\) of \(G\), the chromatic number of~\(H\) is equals to the size of its largest clique.
Our proof follows the same strategy as the proof of Conjecture~\ref{con:diperfect} for this class of digraphs provided by Berge~\cite{Berge1982b}, which we present next.

Let \(D\) be a digraph whose underlying graph is perfect.
Since perfection is an hereditary property for induced subgraphs, i.e., every induced subgraph of a perfect graph is also perfect, to show that \(D\) is \(\alpha\)-diperfect, we only need to show that \(D\) satisfy the \(\alpha\)-property.
A \emph{clique partition} \(\sQ\) of \(D\) is a partition of \(V(D)\) such that each \(Q \in \sQ\) is a clique.
Let \(S\) be a maximum stable set of \(D\) and let \(\sC\) be a clique partition of \(D\) with the smallest size.
From the definition of perfect graphs and by Theorem~\ref{the:weak-perfect}, it follows that the stability number of \(D\) is equal to the smallest size of a clique partition of \(D\), and hence \(|S| = |\sC|\).
Thus, for every clique \(C \in \sC\), it follows that \(|C \cap S| = 1\) and, by Theorem~\ref{the:redei}, there exists a Hamilton path \(P_C\) in~\(D[C]\).
Therefore,  \(\{P_C \:C \in \sC\}\) is a \(S\)-path partition of \(D\), and hence the digraph \(D\) satisfy the \(\alpha\)-property.

\begin{theorem}[Lovász, 1972~\cite{Lovasz72}]\label{the:weak-perfect}
  A graph is perfect if and only if its complement is perfect.
\end{theorem}

\begin{theorem}[Rédei, 1934~\cite{Redei1934}]\label{the:redei}
	Every tournament has a Hamilton path.
\end{theorem}

The only point in our proof for Conjecture~\ref{con:bed} which differs from Berge's proof for Conjecture~\ref{con:diperfect} is that we cannot use Theorem~\ref{the:redei}, because Conjecture~\ref{con:bed} requires the paths starting or ending in the stable set.
Therefore, we prove something stronger than Theorem~\ref{the:redei} for semicomplete digraphs in~\(\fD\).
More precisely, we prove that, for every semicomplete digraph \(D \in \fD\) (i.e., a semicomplete digraph without transitive triangles), there exists a Hamilton path starting or ending in every vertex \(v \in V(D)\).
The next two results describes the structure of semicomplete digraphs without transitive triangles.

We say that a path \(P = x_1x_2 \cdots x_\ell\) is an \emph{\(\{s, t\}\)-path} if \(\{x_1, x_\ell\} = \{s, t\}\).

\begin{figure}
	\centering\scalebox{0.9}{\tikzset{middlearrow/.style={
	decoration={markings,
		mark= at position 0.6 with {\arrow{#1}},
	},
	postaction={decorate}
}}

\tikzset{shortdigon/.style={
	decoration={markings,
		mark= at position 0.5 with {\arrow[inner sep=10pt]{<}},
		mark= at position 0.7 with {\arrow[inner sep=10pt]{>}},
	},
	postaction={decorate}
}}

\tikzset{digon/.style={
	decoration={markings,
		mark= at position 0.40 with {\arrow[inner sep=10pt]{<}},
		mark= at position 0.7 with {\arrow[inner sep=10pt]{>}},
	},
	postaction={decorate}
}}

\tikzset{digon2/.style={
	decoration={markings,
		mark= at position 0.40 with {\arrow[inner sep=10pt]{<}},
		mark= at position 0.65 with {\arrow[inner sep=10pt]{>}},
	},
	postaction={decorate}
}}

\begin{tikzpicture}[scale = 0.6]

	\node (x) [] at (0,0) {};
	\node (a) [black vertex] at ($(x)+(180:2cm)$) {};
	\node (b) [black vertex] at ($(x)+(90:2cm)$) {};
	\node (c) [black vertex] at ($(x)+(0:2cm)$) {};
	\node (d) [black vertex] at ($(x)+(-90:2cm)$) {};
	
	\draw[edge, middlearrow={>}] (b) -- (a);
	\draw[edge, middlearrow={>}] (a) -- (d);
	\draw[edge, middlearrow={>}] (d) -- (c);
	\draw[edge, middlearrow={>}] (c) -- (b);

	\draw[edge, digon] (a) -- (c);
	\draw[edge, digon2] (b) to [in=0, out=0, looseness=2] (d);

\end{tikzpicture}

 }
	\caption{
    The only strong semicomplete digraph free of induced transitive triangles which does not contain a Hamilton \(\{s,t\}\)-path for every  pair of vertices \(s,t\).
    We use an arrow with two heads in opposite directions joining vertices \(u\) and \(v\) to denote that both arcs \(uv\) and \(vu\) belongs to the digraph.}\label{fig:k4-bad}
\end{figure}
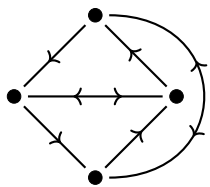

\begin{lemma}\label{lem:semicomplete-strong-characterization}
	Let $D$ be a strong semicomplete digraph without induced transitive triangles.
	If \(D\) is not isomorphic to the digraph in Figure~\ref{fig:k4-bad}, then  there exists a Hamilton \(\{s,t\}\)-path for every pair of vertices \(s, t \in V(D)\).
\end{lemma}
\begin{proof}
	Let \(D\) be a digraph as in the statement, let \(s\) and \(t\) be two vertices of \(D\), and suppose that there is no Hamilton \(\{s,t\}\)-path in \(D\).
	We are going to show that \(D\) is isomorphic to the digraph depicted in Figure~\ref{fig:k4-bad}.
	Since \(D\) is semicomplete, \(s\) and \(t\) are adjacent, and hence there exists an \(\{s,t\}\)-path in \(D\).
	Let $P = v_1v_2 \cdots v_\ell$ be a longest \(\{s,t\}\)-path in \(D\).
	Let $B = V(D) \setminus V(P)$, and note that \(B \neq \emptyset\), since \(P\) is not a Hamilton path.
	Let \(B'= \{u \in B \: V(P) \mapsto u\}\), $B'' = \{u \in B \: u \mapsto V(P)\}$, and $B^* = B \setminus (B' \cup B'')$.
	Clearly \(B', B'', B^*\) are pairwise disjoint.

	\newcommand{\DD}[1]{\textcolor{magenta}{#1}}

	Towards a contradiction, suppose that \(B^* = \emptyset\).
	Since \(D\) is strong and \(B \neq \emptyset\), we have \(B' \neq \emptyset\)  and \(B'' \neq \emptyset\), and hence let \(u \in B'\) and \(w \in B''\).
	Since \(w \mapsto v_1\), \(v_1 \mapsto u\), and \(D[\{u, v_1, w\}]\) is not a transitive triangle, \(uw \in A(D)\), and hence \(v_1uwv_2 v_3 \cdots v_\ell\) is an \(\{s, t\}\)-path longer than~\(P\) (see Figure~\ref{fig:m1}), a contradiction.
	Therefore, \(B^* \neq \emptyset\), and let \(u\) be a vertex in \(B^*\) and let $k = \max{\Set{i: uv_i \in A(D)}}$.
	Note that \(v_j \mapsto u\) for every \(j > k\), and hence \(v_jv_{j - 1} \in A(D)\) for every \(j > k + 1\), otherwise \(D[\{v_{j -1}, v_j, u\}]\) would be a transitive triangle.
	
	Now, we claim that \(u \mapsto v_j\) for every \(j < k\).
	The proof follows by induction on $q = k - j$.
	If $q = 1$, then $j = k - 1$.
	If $v_{k -1}u \in A(D)$, then \(v_1v_2 \cdots v_{k - 1}uv_kv_{k + 1} \cdots v_\ell\) is an $\{s, t\}$-path longer than \(P\), a contradiction.
	Thus, \(u \mapsto v_{k - 1}\).
	Consider now $q > 1$.
	By the induction hypothesis, $u \mapsto v_{j + 1}$.
	If $v_ju \in A(D)$ then \(v_1v_2 \cdots v_juv_{j + 1}v_{j + 2} \cdots v_\ell\) is an $\{s, t\}$-path longer than \(P\), a contradiction.
	Therefore, \(u \mapsto v_j\) and the claim follows.
	Thus, \(v_jv_{j -1} \in A(D)\) for every \(j < k\), otherwise \(D[\{v_{j - 1}, v_j, u\}]\) would be a transitive triangle.

	We may assume that \(k  > 1\), otherwise \(k = 1\) and \(v_\ell v_{\ell - 1} \cdots v_{k + 1} u v_k\) is an \(\{s, t\}\)-path longer than \(P\) (see Figure~\ref{fig:may-2}), a contradiction.
	Also, we may assume that \(k < \ell\), otherwise \(k = \ell\) and, since \(u \mapsto v_j\) for every \(j < k\) and \(u \in B^*\), we have \(v_ku \in A(D)\), and therefore \(v_kuv_{k - 1}v_{k - 2} \cdots v_1\) is an \(\{s, t\}\)-path longer than \(P\) (see Figure~\ref{fig:m3}), a contradiction.
	Therefore, the vertices \(v_{k - 1}\) and \(v_{k + 1}\) exist.

	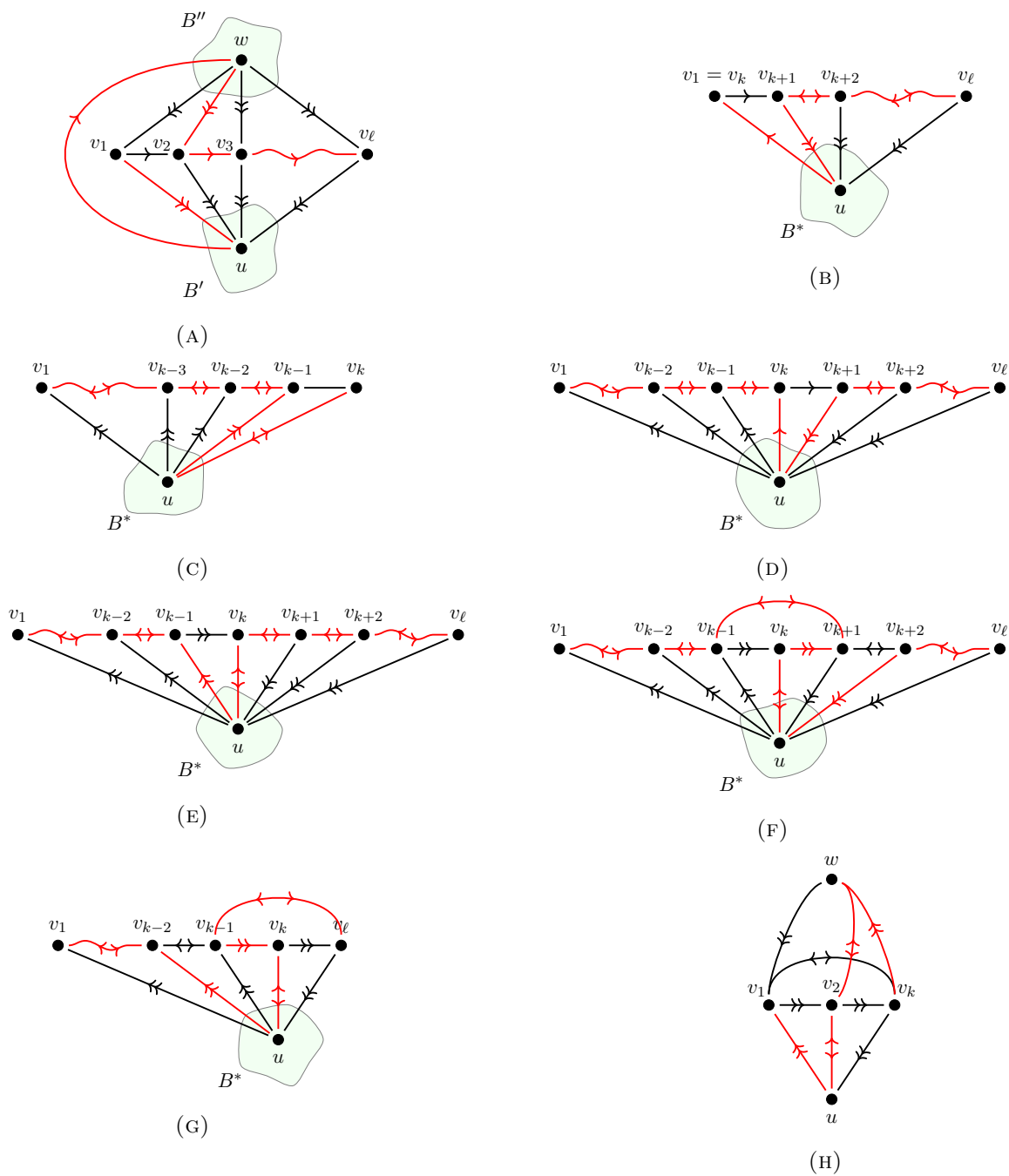
\begin{figure}
		\begin{subfigure}{0.4\textwidth}
			\centering\scalebox{0.8}{\tikzset{middlearrow/.style={
	decoration={markings,
		mark= at position 0.6 with {\arrow{#1}},
	},
	postaction={decorate}
}}

\tikzset{digon/.style={
	decoration={markings,
		mark= at position 0.47 with {\arrow[inner sep=10pt]{<}},
		mark= at position 0.54 with {\arrow[inner sep=10pt]{>}},
	},
	postaction={decorate}
}}

\begin{tikzpicture}[scale = 0.6]

	\def \n{10}
	\def \radiusrate {5};
	\def \theradius {1.3cm};

	\begin{scope}[yshift=3cm]
		\foreach \i in {1,...,\n} {
		  \coordinate (A \i) at
		  ($(0,0)!{1 + .2*rand}!({\i*360/\n}:\theradius)$);
		};

		\draw [fill=setfilling, draw=setborder] plot [smooth cycle] coordinates
		{(A 1) (A 2) (A 3) (A 4) (A 5) (A 6) (A 7) (A 8) (A 9) (A 10)};
		
		\node (label_set) [] at (140:2cm) {$B''$};
	\end{scope}

	\begin{scope}[yshift=-3cm]
		\foreach \i in {1,...,\n} {
		  \coordinate (A \i) at
		  ($(0,0)!{1 + .2*rand}!({\i*360/\n}:\theradius)$);
		};

		\draw [fill=setfilling, draw=setborder] plot [smooth cycle] coordinates
		{(A 1) (A 2) (A 3) (A 4) (A 5) (A 6) (A 7) (A 8) (A 9) (A 10)};
		
		\node (label_set) [] at (-140:2cm) {$B'$};
	\end{scope}

	\node (v3) [black vertex] at (0,0) {};
	\node (v2) [black vertex] at ($(v3)+(180:2cm)$) {};
	\node (v1) [black vertex] at ($(v2)+(180:2cm)$) {};
	\node (vl) [black vertex] at ($(v3)+(0:4cm)$) {};
	
	\node (w) [black vertex] at ($(v3)+(90:3cm)$) {};
	\node (u) [black vertex] at ($(v3)+(-90:3cm)$) {};
	
	\node (label_v1) [] at ($(v1)+(150:.6)$) {$v_1$};
	\node (label_v2) [] at ($(v2)+(150:.6)$) {$v_2$};
	\node (label_v3) [] at ($(v3)+(150:.6)$) {$v_3$};
	\node (label_vl) [] at ($(vl)+(90:.6)$) {$v_\ell$};
	\node (label_w) [] at ($(w)+(90:.6)$) {$w$};
	\node (label_u) [] at ($(u)+(-90:.6)$) {$u$};
	
	\draw[edge, color1, middlearrow={>>}] (v1) -- (u);
	\draw[edge, middlearrow={>>}] (v2) -- (u);
	\draw[edge, middlearrow={>>}] (v3) -- (u);
	\draw[edge, middlearrow={>>}] (vl) -- (u);

	\draw[edge, middlearrow={>>}] (w) -- (v1);
	\draw[edge, color1, middlearrow={>>}] (w) -- (v2);
	\draw[edge, middlearrow={>>}] (w) -- (v3);
	\draw[edge, middlearrow={>>}] (w) -- (vl);

	\draw[edge, middlearrow={>}] (v1) -- (v2);
	\draw[edge, color1, middlearrow={>}] (v2) -- (v3);
	
	\draw[edge, color1, middlearrow={>}] (u) to [in=-180, out=180, looseness=3] (w);

	\draw[edge, color1, fixed point arithmetic, decoration={snake}, segment length=10mm, decorate, postaction={decoration={markings,mark=at position 0.45 with {\arrow{>}}},decorate}] (v3) -- (vl);
\end{tikzpicture}

 }
			\caption{}\label{fig:m1}
		\end{subfigure}
    \hfill
		\begin{subfigure}{0.4\textwidth}
			\centering\scalebox{0.8}{\tikzset{middlearrow/.style={
	decoration={markings,
		mark= at position 0.6 with {\arrow{#1}},
	},
	postaction={decorate}
}}

\tikzset{digon/.style={
	decoration={markings,
		mark= at position 0.4 with {\arrow[inner sep=10pt]{<}},
		mark= at position 0.8 with {\arrow[inner sep=10pt]{>}},
	},
	postaction={decorate}
}}

\begin{tikzpicture}[scale = 0.6]

	\def \n{10}
	\def \radiusrate {5};
	\def \theradius {1.3cm};

	\begin{scope}[yshift=-3cm]
		\foreach \i in {1,...,\n} {
		  \coordinate (A \i) at
		  ($(0,0)!{1 + .2*rand}!({\i*360/\n}:\theradius)$);
		};

		\draw [fill=setfilling, draw=setborder] plot [smooth cycle] coordinates
		{(A 1) (A 2) (A 3) (A 4) (A 5) (A 6) (A 7) (A 8) (A 9) (A 10)};
		
		\node (label_set) [] at (-140:2cm) {$B^*$};
	\end{scope}

	\node (v3) [black vertex] at (0,0) {};
	\node (v2) [black vertex] at ($(v3)+(180:2cm)$) {};
	\node (v1) [black vertex] at ($(v2)+(180:2cm)$) {};
	\node (vl) [black vertex] at ($(v3)+(0:4cm)$) {};
	
	\node (u) [black vertex] at ($(v3)+(-90:3cm)$) {};

	\node (label_v1) [] at ($(v1)+(90:.6)$) {$v_1 = v_k$};
	\node (label_v2) [] at ($(v2)+(90:.6)$) {$v_{k + 1}$};
	\node (label_v3) [] at ($(v3)+(90:.6)$) {$v_{k + 2}$};
	\node (label_vl) [] at ($(vl)+(90:.6)$) {$v_\ell$};
	
	\node (label_u) [] at ($(u)+(-90:.6)$) {$u$};
	
	\draw[edge, color1, middlearrow={>}] (u) -- (v1);
	\draw[edge, color1, middlearrow={>>}] (v2) -- (u);
	\draw[edge, middlearrow={>>}] (v3) -- (u);
	\draw[edge, middlearrow={>>}] (vl) -- (u);

	\draw[edge, middlearrow={>}] (v1) -- (v2);
	\draw[edge, color1, digon] (v2) -- (v3);

	\draw[edge, 
		color1,
		fixed point arithmetic, 
		decoration={snake}, 
		segment length=10mm, 
		decorate, 
		postaction={
			decoration={
				markings,
				mark=at position 0.45 with {\arrow{<}},
				mark=at position 0.6 with {\arrow{>}}},
				decorate}
	] (v3) -- (vl);
\end{tikzpicture}

 }
			\caption{}\label{fig:may-2}
		\end{subfigure}
	
		\begin{subfigure}{0.4\textwidth}
			\centering\scalebox{0.8}{\tikzset{middlearrow/.style={
	decoration={markings,
		mark= at position 0.6 with {\arrow{#1}},
	},
	postaction={decorate}
}}

\tikzset{shortdigon/.style={
	decoration={markings,
		mark= at position 0.45 with {\arrow[inner sep=10pt]{<}},
		mark= at position 0.75 with {\arrow[inner sep=10pt]{>}},
	},
	postaction={decorate}
}}

\tikzset{digon/.style={
	decoration={markings,
		mark= at position 0.47 with {\arrow[inner sep=10pt]{<}},
		mark= at position 0.54 with {\arrow[inner sep=10pt]{>}},
	},
	postaction={decorate}
}}

\begin{tikzpicture}[scale = 0.6]

	\def \n{10}
	\def \radiusrate {5};
	\def \theradius {1.3cm};

	\begin{scope}[yshift=-3cm]
		\foreach \i in {1,...,\n} {
		  \coordinate (A \i) at
		  ($(0,0)!{1 + .2*rand}!({\i*360/\n}:\theradius)$);
		};

		\draw [fill=setfilling, draw=setborder] plot [smooth cycle] coordinates
		{(A 1) (A 2) (A 3) (A 4) (A 5) (A 6) (A 7) (A 8) (A 9) (A 10)};
		
		\node (label_set) [] at (-140:2cm) {$B^*$};
	\end{scope}

	\node (vk3) [black vertex] at (0,0) {};
	\node (vk2) [black vertex] at ($(vk3)+(0:2cm)$) {};
	\node (vk1) [black vertex] at ($(vk2)+(0:2cm)$) {};
	\node (vk) [black vertex] at ($(vk1)+(0:2cm)$) {};
	\node (v1) [black vertex] at ($(vk3)+(180:4cm)$) {};
	
	\node (u) [black vertex] at ($(vk3)+(-90:3cm)$) {};
	
	\node (label_vk3) [] at ($(vk3)+(90:.6)$) {$v_{k-3}$};
	\node (label_vk2) [] at ($(vk2)+(90:.6)$) {$v_{k-2}$};
	\node (label_vk1) [] at ($(vk1)+(90:.6)$) {$v_{k-1}$};
	\node (label_vk) [] at ($(vk)+(90:.6)$) {$v_k$};
	\node (label_v1) [] at ($(v1)+(90:.6)$) {$v_1$};
	\node (label_u) [] at ($(u)+(-90:.6)$) {$u$};

	\draw[edge, middlearrow={>>}] (u) -- (v1);
	\draw[edge, middlearrow={>>}] (u) -- (vk3);
	\draw[edge, middlearrow={>>}] (u) -- (vk2);
	\draw[edge, color1, middlearrow={>>}] (u) -- (vk1);
	\draw[edge, color1, digon] (u) -- (vk);

	\draw[edge] (vk1) -- (vk);
	\draw[edge, color1, shortdigon] (vk2) -- (vk1);
	\draw[edge, color1, shortdigon] (vk3) -- (vk2);

	\draw[edge, 
		color1, 
		fixed point arithmetic, 
		decoration={snake}, 
		segment length=10mm, 
		decorate, 
		postaction={
			decoration={markings,
				mark=at position 0.42 with {\arrow{<}},
				mark=at position 0.57 with {\arrow{>}}
			},
			decorate}] (v1) -- (vk3);

\end{tikzpicture}

 }
			\caption{}\label{fig:m3}
		\end{subfigure}
    \hfill
		\begin{subfigure}{0.5\textwidth}
			\centering\scalebox{0.8}{\tikzset{middlearrow/.style={
	decoration={markings,
		mark= at position 0.6 with {\arrow{#1}},
	},
	postaction={decorate}
}}

\tikzset{shortdigon/.style={
	decoration={markings,
		mark= at position 0.45 with {\arrow[inner sep=10pt]{<}},
		mark= at position 0.75 with {\arrow[inner sep=10pt]{>}},
	},
	postaction={decorate}
}}

\tikzset{digon/.style={
	decoration={markings,
		mark= at position 0.47 with {\arrow[inner sep=10pt]{<}},
		mark= at position 0.54 with {\arrow[inner sep=10pt]{>}},
	},
	postaction={decorate}
}}

\begin{tikzpicture}[scale = 0.6]

	\def \n{10}
	\def \radiusrate {5};
	\def \theradius {1.3cm};

	\begin{scope}[yshift=-3cm]
		\foreach \i in {1,...,\n} {
		  \coordinate (A \i) at
		  ($(0,0)!{1 + .2*rand}!({\i*360/\n}:\theradius)$);
		};

		\draw [fill=setfilling, draw=setborder] plot [smooth cycle] coordinates
		{(A 1) (A 2) (A 3) (A 4) (A 5) (A 6) (A 7) (A 8) (A 9) (A 10)};
		
		\node (label_set) [] at (-140:2cm) {$B^*$};
	\end{scope}

	\node (vk) [black vertex] at (0,0) {};
	\node (vk+1) [black vertex] at ($(vk)+(0:2cm)$) {};
	\node (vk+2) [black vertex] at ($(vk+1)+(0:2cm)$) {};
	\node (vl) [black vertex] at ($(vk)+(0:7cm)$) {};
	\node (vk-1) [black vertex] at ($(vk)+(180:2cm)$) {};
	\node (vk-2) [black vertex] at ($(vk-1)+(180:2cm)$) {};
	\node (v1) [black vertex] at ($(vk)+(180:7cm)$) {};
	
	\node (u) [black vertex] at ($(vk)+(-90:3cm)$) {};
	
	\node (label_vk) [] at ($(vk)+(90:.6)$) {$v_k$};
	\node (label_vk+1) [] at ($(vk+1)+(90:.6)$) {$v_{k+1}$};
	\node (label_vk+2) [] at ($(vk+2)+(90:.6)$) {$v_{k+2}$};
	\node (label_vl) [] at ($(vl)+(90:.6)$) {$v_\ell$};
	\node (label_vk-1) [] at ($(vk-1)+(90:.6)$) {$v_{k-1}$};
	\node (label_vk-2) [] at ($(vk-2)+(90:.6)$) {$v_{k-2}$};
	\node (label_v1) [] at ($(v1)+(90:.6)$) {$v_1$};
	\node (label_u) [] at ($(u)+(-90:.6)$) {$u$};

	\draw[edge, middlearrow={>>}] (u) -- (v1);
	\draw[edge, middlearrow={>>}] (u) -- (vk-2);
	\draw[edge, middlearrow={>>}] (u) -- (vk-1);
	\draw[edge, color1, middlearrow={>}] (u) -- (vk);
	\draw[edge, color1, middlearrow={>>}] (vk+1) -- (u);
	\draw[edge, middlearrow={>>}] (vk+2) -- (u);
	\draw[edge, middlearrow={>>}] (vl) -- (u);

	\draw[edge, color1, shortdigon] (vk-2) -- (vk-1);
	\draw[edge, color1, shortdigon] (vk-1) -- (vk);
	\draw[edge, middlearrow={>}] (vk) -- (vk+1);
	\draw[edge, color1, shortdigon] (vk+1) -- (vk+2);

	\draw[edge, 
		color1, 
		fixed point arithmetic, 
		decoration={snake}, 
		segment length=10mm, 
		decorate, 
		postaction={
			decoration={markings,
				mark=at position 0.5 with {\arrow{<}},
				mark=at position 0.7 with {\arrow{>}}
			},
			decorate}] (v1) -- (vk-2);

	\draw[edge, 
		color1, 
		fixed point arithmetic, 
		decoration={snake}, 
		segment length=10mm, 
		decorate, 
		postaction={
			decoration={markings,
				mark=at position 0.4 with {\arrow{<}},
				mark=at position 0.6 with {\arrow{>}}
			},
			decorate}] (vk+2) -- (vl);

\end{tikzpicture}

 }
			\caption{}\label{fig:m4}
		\end{subfigure}

		\begin{subfigure}{0.4\textwidth}
			\centering\scalebox{0.8}{\tikzset{middlearrow/.style={
	decoration={markings,
		mark= at position 0.6 with {\arrow{#1}},
	},
	postaction={decorate}
}}

\tikzset{shortdigon/.style={
	decoration={markings,
		mark= at position 0.45 with {\arrow[inner sep=10pt]{<}},
		mark= at position 0.75 with {\arrow[inner sep=10pt]{>}},
	},
	postaction={decorate}
}}

\tikzset{digon/.style={
	decoration={markings,
		mark= at position 0.5 with {\arrow[inner sep=10pt]{<}},
		mark= at position 0.7 with {\arrow[inner sep=10pt]{>}},
	},
	postaction={decorate}
}}

\begin{tikzpicture}[scale = 0.6]

	\def \n{10}
	\def \radiusrate {5};
	\def \theradius {1.3cm};

	\begin{scope}[yshift=-3cm]
		\foreach \i in {1,...,\n} {
		  \coordinate (A \i) at
		  ($(0,0)!{1 + .2*rand}!({\i*360/\n}:\theradius)$);
		};

		\draw [fill=setfilling, draw=setborder] plot [smooth cycle] coordinates
		{(A 1) (A 2) (A 3) (A 4) (A 5) (A 6) (A 7) (A 8) (A 9) (A 10)};
		
		\node (label_set) [] at (-140:2cm) {$B^*$};
	\end{scope}

	\node (vk) [black vertex] at (0,0) {};
	\node (vk+1) [black vertex] at ($(vk)+(0:2cm)$) {};
	\node (vk+2) [black vertex] at ($(vk+1)+(0:2cm)$) {};
	\node (vl) [black vertex] at ($(vk)+(0:7cm)$) {};
	\node (vk-1) [black vertex] at ($(vk)+(180:2cm)$) {};
	\node (vk-2) [black vertex] at ($(vk-1)+(180:2cm)$) {};
	\node (v1) [black vertex] at ($(vk)+(180:7cm)$) {};
	
	\node (u) [black vertex] at ($(vk)+(-90:3cm)$) {};
	
	\node (label_vk) [] at ($(vk)+(90:.6)$) {$v_k$};
	\node (label_vk+1) [] at ($(vk+1)+(90:.6)$) {$v_{k+1}$};
	\node (label_vk+2) [] at ($(vk+2)+(90:.6)$) {$v_{k+2}$};
	\node (label_vl) [] at ($(vl)+(90:.6)$) {$v_\ell$};
	\node (label_vk-1) [] at ($(vk-1)+(90:.6)$) {$v_{k-1}$};
	\node (label_vk-2) [] at ($(vk-2)+(90:.6)$) {$v_{k-2}$};
	\node (label_v1) [] at ($(v1)+(90:.6)$) {$v_1$};
	\node (label_u) [] at ($(u)+(-90:.6)$) {$u$};

	\draw[edge, middlearrow={>>}] (u) -- (v1);
	\draw[edge, middlearrow={>>}] (u) -- (vk-2);
	\draw[edge, color1, middlearrow={>>}] (u) -- (vk-1);
	\draw[edge, color1, digon] (u) -- (vk);
	\draw[edge, middlearrow={>>}] (vk+1) -- (u);
	\draw[edge, middlearrow={>>}] (vk+2) -- (u);
	\draw[edge, middlearrow={>>}] (vl) -- (u);

	\draw[edge, color1, shortdigon] (vk-2) -- (vk-1);
	\draw[edge, middlearrow={>>}] (vk-1) -- (vk);
	\draw[edge, color1, shortdigon] (vk) -- (vk+1);
	\draw[edge, color1, shortdigon] (vk+1) -- (vk+2);

	\draw[edge, 
		color1, 
		fixed point arithmetic, 
		decoration={snake}, 
		segment length=10mm, 
		decorate, 
		postaction={
			decoration={markings,
				mark=at position 0.5 with {\arrow{<}},
				mark=at position 0.7 with {\arrow{>}}
			},
			decorate}] (v1) -- (vk-2);

	\draw[edge, 
		color1, 
		fixed point arithmetic, 
		decoration={snake}, 
		segment length=10mm, 
		decorate, 
		postaction={
			decoration={markings,
				mark=at position 0.4 with {\arrow{<}},
				mark=at position 0.6 with {\arrow{>}}
			},
			decorate}] (vk+2) -- (vl);

\end{tikzpicture}

 }
			\caption{}\label{fig:m5}
		\end{subfigure}
    \hfill
		\begin{subfigure}{0.5\textwidth}
			\centering\scalebox{0.8}{\tikzset{middlearrow/.style={
	decoration={markings,
		mark= at position 0.6 with {\arrow{#1}},
	},
	postaction={decorate}
}}

\tikzset{shortdigon/.style={
	decoration={markings,
		mark= at position 0.45 with {\arrow[inner sep=10pt]{<}},
		mark= at position 0.75 with {\arrow[inner sep=10pt]{>}},
	},
	postaction={decorate}
}}

\tikzset{digon/.style={
	decoration={markings,
		mark= at position 0.4 with {\arrow[inner sep=10pt]{<}},
		mark= at position 0.6 with {\arrow[inner sep=10pt]{>}},
	},
	postaction={decorate}
}}

\begin{tikzpicture}[scale = 0.6]

	\def \n{10}
	\def \radiusrate {5};
	\def \theradius {1.3cm};

	\begin{scope}[yshift=-3cm]
		\foreach \i in {1,...,\n} {
		  \coordinate (A \i) at
		  ($(0,0)!{1 + .2*rand}!({\i*360/\n}:\theradius)$);
		};

		\draw [fill=setfilling, draw=setborder] plot [smooth cycle] coordinates
		{(A 1) (A 2) (A 3) (A 4) (A 5) (A 6) (A 7) (A 8) (A 9) (A 10)};
		
		\node (label_set) [] at (-140:2cm) {$B^*$};
	\end{scope}

	\node (vk) [black vertex] at (0,0) {};
	\node (vk+1) [black vertex] at ($(vk)+(0:2cm)$) {};
	\node (vk+2) [black vertex] at ($(vk+1)+(0:2cm)$) {};
	\node (vl) [black vertex] at ($(vk)+(0:7cm)$) {};
	\node (vk-1) [black vertex] at ($(vk)+(180:2cm)$) {};
	\node (vk-2) [black vertex] at ($(vk-1)+(180:2cm)$) {};
	\node (v1) [black vertex] at ($(vk)+(180:7cm)$) {};
	
	\node (u) [black vertex] at ($(vk)+(-90:3cm)$) {};
	
	\node (label_vk) [] at ($(vk)+(90:.6)$) {$v_k$};
	\node (label_vk+1) [] at ($(vk+1)+(90:.6)$) {$v_{k+1}$};
	\node (label_vk+2) [] at ($(vk+2)+(90:.6)$) {$v_{k+2}$};
	\node (label_vl) [] at ($(vl)+(90:.6)$) {$v_\ell$};
	\node (label_vk-1) [] at ($(vk-1)+(90:.6)$) {$v_{k-1}$};
	\node (label_vk-2) [] at ($(vk-2)+(90:.6)$) {$v_{k-2}$};
	\node (label_v1) [] at ($(v1)+(90:.6)$) {$v_1$};
	\node (label_u) [] at ($(u)+(-90:.6)$) {$u$};

	\draw[edge, middlearrow={>>}] (u) -- (v1);
	\draw[edge, middlearrow={>>}] (u) -- (vk-2);
	\draw[edge, middlearrow={>>}] (u) -- (vk-1);
	\draw[edge, color1, digon] (u) -- (vk);
	\draw[edge, middlearrow={>>}] (vk+1) -- (u);
	\draw[edge, color1, middlearrow={>>}] (vk+2) -- (u);
	\draw[edge, middlearrow={>>}] (vl) -- (u);

	\draw[edge, color1, shortdigon] (vk-2) -- (vk-1);
	\draw[edge, middlearrow={>>}] (vk-1) -- (vk);
	\draw[edge, color1, middlearrow={>>}] (vk) -- (vk+1);
	\draw[edge, shortdigon] (vk+1) -- (vk+2);

	\draw[edge, color1, digon] (vk-1) to [in=90, out=90, looseness=1] (vk+1);

	\draw[edge,
		color1,
		fixed point arithmetic, 
		decoration={snake}, 
		segment length=10mm, 
		decorate, 
		postaction={
			decoration={markings,
				mark=at position 0.5 with {\arrow{<}},
				mark=at position 0.7 with {\arrow{>}}
			},
			decorate}] (v1) -- (vk-2);

	\draw[edge, 
		color1, 
		fixed point arithmetic, 
		decoration={snake}, 
		segment length=10mm, 
		decorate, 
		postaction={
			decoration={markings,
				mark=at position 0.4 with {\arrow{<}},
				mark=at position 0.6 with {\arrow{>}}
			},
			decorate}] (vk+2) -- (vl);

\end{tikzpicture}

 }
			\caption{}\label{fig:m6}
		\end{subfigure}
	
		\begin{subfigure}{0.4\textwidth}
			\centering\scalebox{0.8}{\tikzset{middlearrow/.style={
	decoration={markings,
		mark= at position 0.6 with {\arrow{#1}},
	},
	postaction={decorate}
}}

\tikzset{shortdigon/.style={
	decoration={markings,
		mark= at position 0.45 with {\arrow[inner sep=10pt]{<}},
		mark= at position 0.75 with {\arrow[inner sep=10pt]{>}},
	},
	postaction={decorate}
}}

\tikzset{digon/.style={
	decoration={markings,
		mark= at position 0.4 with {\arrow[inner sep=10pt]{<}},
		mark= at position 0.6 with {\arrow[inner sep=10pt]{>}},
	},
	postaction={decorate}
}}

\begin{tikzpicture}[scale = 0.6]

	\def \n{10}
	\def \radiusrate {5};
	\def \theradius {1.3cm};

	\begin{scope}[yshift=-3cm]
		\foreach \i in {1,...,\n} {
		  \coordinate (A \i) at
		  ($(0,0)!{1 + .2*rand}!({\i*360/\n}:\theradius)$);
		};

		\draw [fill=setfilling, draw=setborder] plot [smooth cycle] coordinates
		{(A 1) (A 2) (A 3) (A 4) (A 5) (A 6) (A 7) (A 8) (A 9) (A 10)};
		
		\node (label_set) [] at (-140:2cm) {$B^*$};
	\end{scope}

	\node (vk) [black vertex] at (0,0) {};
	\node (vl) [black vertex] at ($(vk)+(0:2cm)$) {};
	\node (vk-1) [black vertex] at ($(vk)+(180:2cm)$) {};
	\node (vk-2) [black vertex] at ($(vk-1)+(180:2cm)$) {};
	\node (v1) [black vertex] at ($(vk)+(180:7cm)$) {};
	
	\node (u) [black vertex] at ($(vk)+(-90:3cm)$) {};
	
	\node (label_vk) [] at ($(vk)+(90:.6)$) {$v_k$};
	\node (label_vl) [] at ($(vl)+(90:.6)$) {$v_\ell$};
	\node (label_vk-1) [] at ($(vk-1)+(90:.6)$) {$v_{k-1}$};
	\node (label_vk-2) [] at ($(vk-2)+(90:.6)$) {$v_{k-2}$};
	\node (label_v1) [] at ($(v1)+(90:.6)$) {$v_1$};
	\node (label_u) [] at ($(u)+(-90:.6)$) {$u$};

	\draw[edge, middlearrow={>>}] (u) -- (v1);
	\draw[edge, color1, middlearrow={>>}] (u) -- (vk-2);
	\draw[edge, middlearrow={>>}] (u) -- (vk-1);
	\draw[edge, color1, digon] (u) -- (vk);
	\draw[edge, middlearrow={>>}] (vl) -- (u);

	\draw[edge, shortdigon] (vk-2) -- (vk-1);
	\draw[edge, color1, middlearrow={>>}] (vk-1) -- (vk);
	\draw[edge, middlearrow={>>}] (vk) -- (vl);

	\draw[edge, color1, digon] (vk-1) to [in=90, out=90, looseness=1] (vl);

	\draw[edge, 
		color1, 
		fixed point arithmetic, 
		decoration={snake}, 
		segment length=10mm, 
		decorate, 
		postaction={
			decoration={markings,
				mark=at position 0.5 with {\arrow{<}},
				mark=at position 0.7 with {\arrow{>}}
			},
			decorate}] (v1) -- (vk-2);
	
\end{tikzpicture}

 }
			\caption{}\label{fig:m7}
		\end{subfigure}
    \hfill
		\begin{subfigure}{0.4\textwidth}
			\centering\scalebox{0.8}{\tikzset{middlearrow/.style={
	decoration={markings,
		mark= at position 0.6 with {\arrow{#1}},
	},
	postaction={decorate}
}}

\tikzset{shortdigon/.style={
	decoration={markings,
		mark= at position 0.5 with {\arrow[inner sep=10pt]{<}},
		mark= at position 0.7 with {\arrow[inner sep=10pt]{>}},
	},
	postaction={decorate}
}}

\tikzset{digon/.style={
	decoration={markings,
		mark= at position 0.4 with {\arrow[inner sep=10pt]{<}},
		mark= at position 0.5 with {\arrow[inner sep=10pt]{>}},
	},
	postaction={decorate}
}}

\begin{tikzpicture}[scale = 0.6]

	\node (v2) [black vertex] at (0,0) {};
	\node (v1) [black vertex] at ($(v2)+(180:2cm)$) {};
	\node (vk) [black vertex] at ($(v2)+(0:2cm)$) {};

	\node (u) [black vertex] at ($(v2)+(-90:3cm)$) {};
	\node (w) [black vertex] at ($(v2)+(90:4cm)$) {};
	
	\node (label_v1) [] at ($(v1)+(130:.6)$) {$v_1$};
	\node (label_v2) [] at ($(v2)+(90:.6)$) {$v_2$};
	\node (label_vk) [] at ($(vk)+(50:.6)$) {$v_k$};
	\node (label_u) [] at ($(u)+(-90:.6)$) {$u$};
	\node (label_w) [] at ($(w)+(90:.6)$) {$w$};

	\draw[edge, color1, middlearrow={>>}] (u) -- (v1);
	\draw[edge, color1, shortdigon] (u) -- (v2);
	\draw[edge, middlearrow={>>}] (vk) -- (u);

	\draw[edge, middlearrow={>>}] (v1) -- (v2);
	\draw[edge, middlearrow={>>}] (v2) -- (vk);

	\draw[edge, digon] (v1) to [in=90, out=90, looseness=1] (vk);

	\draw[edge, middlearrow={>>}] (w) to [in=90, out=200, looseness=0.5] (v1);
	\draw[edge, color1, middlearrow={>>}] (vk) to [in=-20, out=90, looseness=0.5] (w);
	\draw[edge, color1, digon] (v2) to [in=-20, out=50, looseness=0.5] (w);

\end{tikzpicture}

 }
			\caption{}\label{fig:may-8}
		\end{subfigure}
		\caption{Illustrations for the proof of Lemma~\ref{lem:semicomplete-strong-characterization}.
      Let \(u\) and \(v\) be two vertices in the picture.
      We use an arrow with two heads from \(u\) to \(v\) to denote that \(u \mapsto v\), and we use an arrow with two heads in opposite directions joining \(u\) and \(v\) to denote that \(u \ \leftrightarrow v\).
      A snaked line joining \(u\) and \(v\) represents a path with possibly several internal vertices.}
	\end{figure}

	We may assume that \(v_kv_{k - 1} \notin A(G)\), otherwise \(v_\ell v_{\ell -1} \cdots v_{k + 1}uv_{k}v_{k - 1}v_{k - 2} \cdots v_1\) would be an \(\{s,t\}\)-path longer than \(P\) (see Figure~\ref{fig:m4}), a contradiction.
	Thus \(B' = B'' = \emptyset\), otherwise \(D[\{v_k, v_{k - 1}, w\}]\) would be a transitive triangle, where \(w \in B' \cup B''\).
	Since, \(u \mapsto v_{k - 1}\), \(v_{k - 1} \mapsto v_k\), and \(D[\{v_{k - 1}, v_k, u\}]\) is not a transitive triangle, we have \(v_ku \in A(D)\), and hence \(v_k \leftrightarrow u\).
	Thus, we have \(v_{k + 1}v_k \notin A(D)\), otherwise \(v_\ell v_{\ell - 1} \cdots v_k u v_{k - 1} v_{k - 2} \cdots v_1\) would be an \(\{s, t\}\)-path longer than \(P\) (see Figure~\ref{fig:m5}), a contradiction.
	Thus, since \(v_{k - 1} \mapsto v_k\), \(v_k \mapsto v_{k + 1}\), and \(D[\{v_{k - 1}, v_k, v_{k + 1}\}]\) and \(D[\{v_{k - 1}, u, v_{k + 1}\}]\) are not transitive triangles, we have \(v_{k +1 } \leftrightarrow v_{k - 1}\), and hence \(D[\{u, v_{k - 1}, v_k, v_{k + 1}\}]\) is isomorphic to the digraph in Figure~\ref{fig:k4-bad}.

	Now, note that \(\ell = k + 1\), otherwise \(\ell > k + 1\) and \(v_\ell v_{\ell - 1} \cdots v_{k + 2} u v_{k} v_{k + 1} v_{k - 1} v_{k - 2} \cdots v_1\) is an \(\{s, t\}\)-path longer than \(P\) (see Figure~\ref{fig:m6}), a contradiction; also, \(k = 2\), otherwise \(k > 2\) and \(v_\ell v_{k - 1} v_k u v_{k - 2} v_{k - 3} \cdots v_1\) is an \(\{s,t\}\)-path longer than \(P\) (see Figure~\ref{fig:m7}), a contradiction.
	Therefore, \(k = 2\), \(\ell = 3\), and \(P = v_1v_2v_3\).
	Moreover, for every vertex \(w \in B^*\), we have \(D[V(P) \cup \{w\}]\) is isomorphic to the digraph depicted in Figure~\ref{fig:k4-bad}.
	Thus, \(B^* = \{u\}\), otherwise there exists a vertex \(w \in B^*\) different of \(u\) and \(v_3 w v_2 u v_1\) would be an \(\{s, t\}\)-path longer than \(P\) (see Figure~\ref{fig:may-8}), a contradiction.
		Therefore, $D = D[V(P) \cup \{u\}]$, which is isomorphic to the digraph in Figure~\ref{fig:k4-bad}, and hence the result follows.
\end{proof}

\begin{lemma}\label{lem:semicomplete-non-strong-characterization}
	Let $D$ be a semicomplete digraph without induced transitive triangles.
	If \(D\) is not strong, then $D$ has exactly two strong components, say $X_1$ and $X_2$, with
	$X_1$ being the minimal strong component, $X_1$ and \(X_2\) are complete digraphs, and for every \(s \in V(X_1)\)  and \(t\in V(X_2)\) there exists a Hamilton path in \(D\) starting at \(s\) and ending at \(t\).
\end{lemma}
\begin{proof}
	Since \(D\) is not strong, it has at least two strong components.
	Towards a contradiction, suppose that \(D\) has more than two strong components and let \(x, y, z\) be vertices from three distinct strong components of \(D\).
	Note that if \(u\) and \(v\) are two adjacent vertices which belong to distinct strong components, then either \(u \mapsto v\) or \(v \mapsto u\).
	Thus \(D[\{x,y,z\}]\) is a transitive triangle, a contradiction.
	Thus \(D\) has precisely two strong components, say \(X_1\) and \(X_2\), and suppose, without loss of generality, that \(X_1\) is the minimal strong component of \(D\).
	Let \(u, v \in V(X_i)\) and let \(w \in V(X_{3 - i})\), for \(i \in \{1, 2\}\).
	Thus either \(w \mapsto \{u, v\}\) or \(\{u, v\} \mapsto w\), and since \(D[\{u, v, w\}]\) is not a transitive triangle, we have \(u \leftrightarrow v\).
	Now the proof follows trivially since \(X_1\) and \(X_2\) are complete digraphs and \(V(X_1) \mapsto V(X_2)\).
\end{proof}

Since every induced subdigraph of a semicomplete digraph is also semicomplete, it  suffices to show that semicomplete digraphs satisfy the BE-property in order to prove that they are BE-diperfect.
Hence, Theorem~\ref{the:semicomplete-be-diperfect} follows directly from Lemmas~\ref{lem:semicomplete-strong-characterization} and~\ref{lem:semicomplete-non-strong-characterization} (note that the exception of Lemma~\ref{lem:semicomplete-strong-characterization}, the digraph in Figure~\ref{fig:k4-bad}, is BE-diperfect).

\begin{theorem}\label{the:semicomplete-be-diperfect}
	If \(D \in \fD\) is a semicomplete digraph (i.e., a semicomplete digraph without induced transitive triangles), then \(D\) is BE-diperfect.
\end{theorem}

We remark that instead of using Lemma~\ref{lem:semicomplete-strong-characterization} in the proof of Theorem~\ref{the:semicomplete-be-diperfect}, we could use a result provided by Camion~\cite{Camion1959} that says that every strong semicomplete digraph has a Hamilton cycle.
However, Lemma~\ref{lem:semicomplete-strong-characterization} provide more structure on the Hamilton path, which may be useful in future works.  
Theorem~\ref{the:bed-perfect-graphs} verifies Conjecture~\ref{con:bed} for digraphs whose underlying graphs are perfect.

\begin{theorem}\label{the:bed-perfect-graphs}
	If \(D \in \fD\) and \(U(D)\) is perfect, then \(D\) is BE-diperfect.
\end{theorem}
\begin{proof}
			  Note that for every induced subdigraph \(D'\) of \(D\), it follows that \(D' \in \fD\) and \(U(D')\) is perfect, and hence, to prove the result it suffices to show that \(D\) satisfies the BE-property.
	Let \(S\) be a maximum stable set of \(U(D)\) and let \(\sC = \{C_1, C_2, \ldots, C_k\}\) be a clique partition of \(U(D)\) with the smallest size.
  By Theorem~\ref{the:weak-perfect}, \(|S| = k\), and hence, for every \(C_i \in \sC\), we have \(C_i \cap S = \{x_i\}\).
	Since \(D[C_i]\) is a semicomplete digraph, \(S^i = \{x_i\}\) is a maximum stable set of \(D[C_i]\), and hence, by Theorem~\ref{the:semicomplete-be-diperfect},
	there exists an \(S^i_\text{BE}\)-path partition \(\{P_i\}\) of \(D[C_i]\).
	Therefore, \(\{P_1, P_2, \ldots, P_k\}\) is an \SBE-path partition of \(D\), and the result follows.
\end{proof}

\section{Technical results}
\label{sec:structural-results}

In this section we present some structural results for  \(\alpha\)-diperfect digraphs and BE-diperfect digraphs.
Given a digraph \(D\), we say a vertex \(v \in V(D)\) is \emph{universal} if \(v\) is adjacent to every other vertex of \(D\).

\begin{lemma}\label{lem:universal-diperfect}
	Let \(D\) be a digraph  and let \(v\) be a universal vertex of \(D\).
	If \(D - v\) satisfies the \(\alpha\)-property, then \(D\) satisfies the \(\alpha\)-property.
\end{lemma}
\begin{proof}
	Let $S$ be a maximum stable set in \(D\).
	We may assume that \(|S| = \alpha(D) \geq 2\), otherwise \(D\) is a semicomplete digraph and the result follows by Theorem~\ref{the:redei}.
	Since \(\alpha(D) \geq 2\) and \(v\) is a universal vertex, it follows that \(v \notin S\).
	Thus \(S\) is a maximum stable set of \(D - v\), and since \(D - v\) satisfies the \(\alpha\)-property by hypothesis, there exists an \(S\)-path partition \(\sP\) of \(D - v\).
    Let \(P \in \sP\) be a path and note that \(v\) is adjacent to every vertex of \(P\).
    Thus, it is not hard to show that there exists a path \(P'\) of \(D\) such that \(V(P') = V(P) \cup \{v\}\).
 	Therefore \((\sP \cup \{P'\}) \setminus \{P\}\) is an \(S\)-path partition of \(D\).
    Since \(S\) was arbitrarily chosen, the result follows.
\end{proof}

\begin{corollary}\label{cor:universal-diperfect}
	Let \(D\) be a digraph and let \(v\) be a universal vertex of \(D\).
	Then \(D\) is \(\alpha\)-diperfect if, and only if, \(D - v\) is \(\alpha\)-diperfect.
\end{corollary}
\begin{proof}
	The necessity follows directly from the definition of \(\alpha\)-diperfect digraphs, thus it only remains to prove the sufficiency.
	Let \(H\) be an induced subdigraph of \(D\).
  We claim that \(H\) satisfy the \(\alpha\)-property. 
	If \(v \notin V(H)\), then \(H\) is an induced subdigraph of \(D - v\), and since \(D - v\) is \(\alpha\)-diperfect, \(H\) satisfies the \(\alpha\)-property.
	Thus suppose that \(v \in V(H)\), and note that \(v\) is a universal vertex in~\(H\).
  We may assume that \(H\) has at least two vertices, otherwise \(H\) satisfy the \(\alpha\)-property in a trivial way.
	Since \(H - v\) is an induced subdigraph of \(D - v\), the digraph \(H - v\) satisfies the \(\alpha\)-property, and hence by Lemma~\ref{lem:universal-diperfect}, \(H\) satisfies the \(\alpha\)-property.
  Since \(H\) was chosen arbitrarily, it follows that the digraph \(D\) is \(\alpha\)-diperfect.
\end{proof}

As a consequence of Corollary~\ref{cor:universal-diperfect}, if \(D\) is a counterexample for Conjecture~\ref{con:diperfect} with the smallest number of vertices, then no vertex of \(D\) is universal.

A digraph \(D\) is \(\alpha\)-diperfect if, and only if, its inverse digraph is \(\alpha\)-diperfect.
The same applies for BE-diperfect digraphs.
This directional duality is useful to fix an orientation for a given path or arc in the middle of a proof. 
We use this fact in the proof of Lemma~\ref{lem:universal-be-diperfect}, which is the analogous version of Lemma~\ref{lem:universal-diperfect} for BE-diperfect digraphs.
We remark that Lemma~\ref{lem:universal-be-diperfect} requires \(D\) to be a digraph in \(\fD\) (see Figure~\ref{fig:example-requirement-be}), unlike Lemma~\ref{lem:universal-diperfect} that does not require \(D\) to be a digraph in \(\fB\).

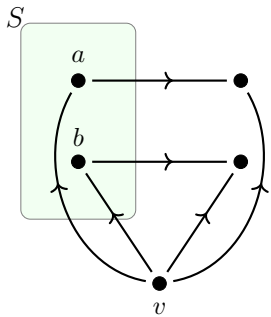
\begin{figure}
	\centering\scalebox{0.9}{\tikzset{middlearrow/.style={
	decoration={markings,
		mark= at position 0.6 with {\arrow{#1}},
	},
	postaction={decorate}
}}

\tikzset{shortdigon/.style={
	decoration={markings,
		mark= at position 0.5 with {\arrow[inner sep=10pt]{<}},
		mark= at position 0.7 with {\arrow[inner sep=10pt]{>}},
	},
	postaction={decorate}
}}

\tikzset{digon/.style={
	decoration={markings,
		mark= at position 0.4 with {\arrow[inner sep=10pt]{<}},
		mark= at position 0.5 with {\arrow[inner sep=10pt]{>}},
	},
	postaction={decorate}
}}

\begin{tikzpicture}[scale = 0.6]

	\node (x) [] at (0,0) {};
	\node (a_fake) [] at ($(x)+(180:2cm)$) {};
	\node (b_fake) [] at ($(a_fake)+(-90:2cm)$) {};
	\draw [rounded corners,draw=setborder,fill=setfilling] ($(a_fake)+(135:2)$) rectangle ($(b_fake)+(-45:2)$);

	\node (v) [black vertex] at ($(x)+(-90:5cm)$) {};
	\node (a) [black vertex] at ($(x)+(180:2cm)$) {};
	\node (b) [black vertex] at ($(a)+(-90:2cm)$) {};
	\node (c) [black vertex] at ($(a)+(0:4cm)$) {};
	\node (d) [black vertex] at ($(b)+(0:4cm)$) {};

	\node (label_S) [] at ($(a)+(135:2.2)$) {\large $S$};

	\node (label_a) [] at ($(a)+(90:.6)$) {$a$};
	\node (label_b) [] at ($(b)+(90:.6)$) {$b$};
	\node (label_v) [] at ($(v)+(-90:.6)$) {$v$};
	
	\draw[edge, middlearrow={>}] (a) -- (c);
	\draw[edge, middlearrow={>}] (b) -- (d);
	\draw[edge, middlearrow={>}] (v) -- (b);
	\draw[edge, middlearrow={>}] (v) -- (d);

	\draw[edge, middlearrow={>}] (v) to [in=-120, out=170, looseness=1] (a);
	\draw[edge, middlearrow={>}] (v) to [in=-60, out=10, looseness=1] (c);

\end{tikzpicture}

 }
	\caption[Example of a digraph \(D\) with a universal vertex \(v\) such that \(D - v\) satisfies the BE-property but \(D\) does not satisfies.]{
		Example of a digraph \(D\) with a universal vertex \(v\) such that \(D - v\) satisfies the BE-property but \(D\) does not.
		Note that $S = \{a, b\}$ is a maximum stable set of \(D\) but $D$ has no $S_\text{BE}$-path partition.
	}\label{fig:example-requirement-be}
\end{figure}

\begin{lemma}\label{lem:universal-be-diperfect}
	Let \(D\) be a digraph in \(\fD\) and let \(v\) be a universal vertex of \(D\).
	If \(D - v\) satisfies the BE-property, then \(D\) satisfies the BE-property.
\end{lemma}
\begin{proof}
	Towards a contradiction, suppose that \(D\) does not satisfy the BE-property.
	Let $S$ be a maximum stable set in \(D\), and suppose that there is no \SBE-path partition of \(D\).
	We may assume that \(\alpha(D) \geq 2\), otherwise, by Theorem~\ref{the:semicomplete-be-diperfect}, \(D\) satisfies the BE-property, a contradiction.
	Since \(\alpha(D) \geq 2\) and \(v\) is a universal vertex, we have \(v \notin S\).
	Thus \(S\) is a maximum stable set of \(D - v\), and since \(D - v\) satisfies the BE-property, there exists an \SBE-path partition \(\sP\) of \(D - v\).
	Let $P = u_1 \cdots u_\ell$ be a path in \(\sP\).
	By the directional duality, we may assume, without loss of generality, that \(u_1 \in S\).
	Since \(v\) is a universal vertex, it is adjacent to every vertex of \(P\).

	Now we show by induction on \(q = \ell - j\) that \(v \mapsto u_j\), for \(j = 1, 2, \ldots, \ell\).
	First, suppose that \(q = 0\),  and hence \(u_j = u_\ell\).
	If \(u_j v \in A(D)\), then  \((\sP \cup \{Pv\}) \setminus \{P\}\) is an \SBE-path partition of~\(D\), a contradiction.
	Otherwise \(v \mapsto u_j\), since \(v\) and \(u_j\) are adjacent, and the claim holds.
	Now suppose that \(q > 0\).
	By the induction hypothesis \(v \mapsto u_i\) for every \(i > j\), in particular \(v \mapsto u_{j + 1}\).
	If \(u_j v \in A(D)\), then let \(P' = u_1 \cdots u_j v u_{j + 1} \cdots u_\ell\), and hence \((\sP \cup \{P'\}) \setminus \{P\}\) is an \SBE-path partition of \(D\), a contradiction.
	Thus \(v \mapsto u_j\) and the claim holds.
	
	For \(i = 1, 2, \ldots, \ell - 1\),  it follows that the induced subdigraph \(D[\{u_i, u_{i + 1}, v\}]\) is not a transitive triangle, since \(D \in \fD\),  and hence \(u_i \leftrightarrow u_{i + 1}\), since \(v \mapsto \{u_i, u_{i + 1}\}\).
 	Therefore \(P' = vu_\ell u_{\ell - 1} \cdots u_1\) is a path of \(D\) and \((\sP \cup \{P'\}) \setminus \{P\}\) is an \SBE-path partition of \(D\), a contradiction.
\end{proof}

Next we present Corollary~\ref{cor:universal-be-diperfect} which is similar to Corollary~\ref{cor:universal-diperfect} but for BE-diperfect digraphs.
The only difference in the proof of the two corollaries is that in Corollary~\ref{cor:universal-be-diperfect} we use Lemma~\ref{lem:universal-be-diperfect} instead of Lemma~\ref{lem:universal-diperfect}, and for this reason, we omit its proof.

\begin{corollary}\label{cor:universal-be-diperfect}
	Let \(D\) be a digraph in \(\fD\) and let \(v\) be a universal vertex of \(D\).
	Then \(D\) is BE-diperfect if, and only if, \(D - v\) is BE-diperfect.
\end{corollary}

Similar to what happened before, as a consequence of Corollary~\ref{cor:universal-be-diperfect}, if \(D\) is a counterexample for Conjecture~\ref{con:bed} with the smallest number of vertices, then no vertex of \(D\) is universal.

The following result, Lemma~\ref{lem:dip-vertex-partition-sum-alpha}, says that if a digraph \(D\) can be partitioned into \(k\) induced subdigraphs, say \(H_1, H_2, \ldots, H_k\), such that \(k \geq 2\), every \(H_i\) satisfies the \(\alpha\)-property, and \(\alpha(D) = \sum_{i = 1}^k \alpha(H_i)\), then \(D\) satisfies the \(\alpha\)-property.
As a consequence of this result, if \(D\) is a counterexample for Conjecture~\ref{con:diperfect} with the smallest number of vertices, then \(D\) does not admit such partition.
This allows us to verify Conjecture~\ref{con:diperfect} by supposing that there exists a counterexample, taking a smallest one, and  then exhibiting such partition, which results in a contradiction. 
Another way of using Lemma~\ref{lem:dip-vertex-partition-sum-alpha} is by exhibiting such partition in such a way that each \(H_i\) belongs to some class of digraphs that satisfies the \(\alpha\)-property; like, for example, perfect digraphs (which includes bipartite digraphs and semicomplete digraphs) or series-parallel digraphs (see Section~\ref{sec:series-parallel}).
We should remark though that not every \(\alpha\)-diperfect digraph \(D\) admits such partition, take for example an (directed) odd cycle.

\begin{lemma}\label{lem:dip-vertex-partition-sum-alpha}
	If a digraph \(D\) can be partitioned into \(k\) induced subdigraphs, say \(H_1,\) \( H_2, \ldots, H_k\), such that  \(k \geq 2\),  every \(H_i\) satisfies the \(\alpha\)-property, and \(\alpha(D) = \sum_{i = 1}^k \alpha(H_i)\), then \(D\) satisfies the \(\alpha\)-property.
\end{lemma}
\begin{proof}
	Let \(S\) be a maximum stable set of \(D\) and let \(S_i = S \cap V(H_i)\) for \(i = 1, 2, \ldots, k\).
	Thus,
	\[\alpha(D) = |S| = \sum_{i = 1}^k |S_i| \leq \sum_{i = 1}^k \alpha(H_i) = \alpha(D).\]
	Hence, \(S_i\) is a maximum stable set of \(H_i\), and since the latter satisfies the \(\alpha\)-property, there exists an \(S_i\)-path partition \(\sP_i\) of \(H_i\), for \(i = 1, \ldots, k\).
	Therefore, \(\bigcup_{i = 1}^k \sP_i\) is an \(S\)-path partition of \(D\), and since \(S\) is an arbitrary maximum stable set of \(D\), the result follows.
\end{proof}

Next lemma is the version of Lemma~\ref{lem:dip-vertex-partition-sum-alpha} for BE-diperfect digraphs.
Everything that we discussed for Lemma~\ref{lem:dip-vertex-partition-sum-alpha} applies for Lemma~\ref{lem:bed-vertex-partition-sum-alpha}.
Moreover, its proof is precisely the same, so we omit it.

\begin{lemma}\label{lem:bed-vertex-partition-sum-alpha}
	If a digraph \(D\) can be partitioned into \(k\) induced subdigraphs, say \(H_1,\) \( H_2, \ldots, H_k\), such that  \(k \geq 2\),  every \(H_i\) satisfies the BE-property, and \(\alpha(D) = \sum_{i = 1}^k \alpha(H_i)\), then \(D\) satisfies the BE-property.
\end{lemma}

\subsection{Partition lemmas}\label{sec:partition-lemmas}

Now we present two results that allow us to partition a digraph \(D\) into two induced subdigraphs \(D_1\) and \(D_2\) such that \(\alpha(D) = \alpha(D_1) + \alpha(D_2)\).
When dealing with a stable set in a digraph, we are, indeed, looking at its underlying graph, since the orientation of the arcs does not matter.
Therefore in order to simplify notation and terminology, not having to worry about the direction of the arcs, we state such results in terms of graphs.
In forthcoming sections, we are going to use these results together with Lemmas~\ref{lem:dip-vertex-partition-sum-alpha} and~\ref{lem:bed-vertex-partition-sum-alpha} to verify Conjectures~\ref{con:diperfect} and~\ref{con:bed}.

A \emph{clique cut} of a connected graph \(G\) is a clique \(X\) of \(G\) such that \(G - X\) is disconnected.
A \emph{cut vertex} of \(G\) is a vertex \(v\) such that \(G - v\) is disconnected.

\begin{lemma}\label{lem:clique-cut-alpha=alpha1+alpha2}
	If $B$ is a clique cut of a graph $G$, then \(G\) can be partitioned into two proper induced subgraphs \(H_1\) and \(H_2\) such that \(\alpha(G) = \alpha(H_1) + \alpha(H_2)\). 
	Moreover, if $uv$ is an edge of~$G$ such that $u \in V(H_1)$ and $v \in V(H_2)$, then $\{u, v\} \cap B \neq \emptyset$.
\end{lemma}
\begin{proof}
	First, note that if $G$ is disconnected,
	then the result follows easily by choosing \(H_1 = C\) and \(H_2 = G - V(C)\), where $C$ is a connected component of $G$.
	Therefore, we may assume that~$G$ is connected.
	The remaining proof follows by induction on $|B|$.
	Let $v \in B$ and $G' = G - v$.
	
	We claim that the result holds for $G'$.
	If $G'$ is disconnected, then the result holds for \(G'\) by our previous discussion.
	Otherwise, $G'$ is connected, $B \setminus \{v\}$ is a clique cut of $G'$, and by the induction hypothesis, the result holds for $G'$.
	Therefore, let \(H'_1\) and \(H'_2\) be a partition of~$G'$ satisfying the lemma's result.
	Let \(H^+_i = G[V(H'_i) \cup \{v\}]\) for each \(i \in \{1, 2\}\).
	Note that $\alpha(H'_i) \leq \alpha(H^+_i) \leq \alpha(H'_i) + 1$, for each \(i \in \{1, 2\}\), and $\alpha(G') \leq \alpha(G) \leq \alpha(G') + 1$.
	The remaining proof is divided into two cases depending on whether $\alpha(G) = \alpha(G')$ or $\alpha(G) = \alpha(G') + 1$.
	
	First suppose that $\alpha(G) = \alpha(G')$.
		Towards a contradiction, suppose that $\alpha(H^+_i) = \alpha(H'_i) + 1$ for each $i \in \{1,2 \}$.
	In this case every maximum stable set of $H^+_i$ contains the vertex $v$ for each \(i \in \{1, 2\}\).
	Let $S^*_i$ be a maximum stable set of $H^+_i$ for each \(i \in \{1, 2\}\).
	We claim that $S^* = S^*_1 \cup S^*_2$ is a stable set of $G$.
	Suppose that there exists a pair of vertices $a,b \in S^*$ such that $ab \in E(G)$.
	Since $S^*_1$ and $S^*_2$ are stable sets of $H^+_1$ and $H^+_2$, respectively, $a \neq v$, $b \neq v$, and we may assume that $a \in V(H'_1)$ and $b \in V(H'_2)$.
	Since $v \in S^*_1$, $v \in S^*_2$, and $B$ is a clique, we have $S^*_1 \cap B = S^*_2 \cap B = \{v\}$.
	Hence $a \notin B$ and $b \notin B$, a contradiction to the fact that the lemma holds for \(G'\).
	Therefore $S^*$ is a stable set of $G$, and hence
	\begin{align*}
	|S^*| &= |S^*_1| + |S^*_2| - 1 = \alpha(H^+_1) + \alpha(H^+_2) - 1\\
	&= \alpha(H'_1) + \alpha(H'_2) + 1 = \alpha(G') + 1\\
	& = \alpha(G) + 1,
	\end{align*}
	a contradiction.
	Therefore, we may assume, without loss of generality, that $\alpha(H^+_1) = \alpha(H'_1)$, and hence
	\[\alpha(G) = \alpha(G') = \alpha(H'_1) + \alpha(H'_2) = \alpha(H^+_1) + \alpha(H'_2),\]
	and the result follows with \(H_1 = H^+_1\) and \(H_2 = H'_2\).
	
	Now suppose that $\alpha(G') = \alpha(G) - 1$.
	Thus every maximum stable set of $G$ contains the vertex~$v$.
	Suppose that $\alpha(H^+_1) = \alpha(H'_1)$.
	Let $S$ be a maximum stable set in $G$ and let $S_1 = S \cap V(H^+_1)$ and $S_2 = S \cap V(H'_2)$.
	Hence $|S| = |S_1| + |S_2| \leq \alpha(H^+_1) + \alpha(H'_2) = \alpha(G') = \alpha(G) - 1$, a contradiction.
	Therefore, $\alpha(H^+_1) = \alpha(H'_1) + 1$, and
	$\alpha(G) = \alpha(G') + 1 = \alpha(H'_1) + \alpha(H'_2) + 1 = \alpha(H^+_1) + \alpha(H'_2)$.
	Thus, the result follows with \(H_1 = H^+_1\) and \(H_2 = H'_2\).
\end{proof}

\begin{lemma}\label{lem:cycle-at-most-two-vertices-degree-greater-2}
	If \(G\) contains a proper induced cycle containing at most two
	vertices of degree greater than two, then \(G\) can be partitioned into two proper
	induced subgraphs \(H_1\) and \(H_2\) such that \(\alpha(G) = \alpha(H_1) +
	\alpha(H_2)\).
\end{lemma}
\begin{proof}	Let \(C\) be a proper induced cycle of \(G\).
	If every vertex of \(C\) has degree equals to two, then~\(C\) is a component of \(G\) and the result holds with \(H_1 = C\) and \(H_2 = G - V(C)\).
	Now, suppose that \(C\) has precisely one vertex, say \(u\), with degree greater than two.
	Thus, \(u\) is a cut vertex of \(G\) and the result follows by Lemma~\ref{lem:clique-cut-alpha=alpha1+alpha2}.
	Hence, we may assume that \(C\) has precisely two vertices, say \(u\) and \(v\), with degree greater than two.
	Let \(C = x_0 \cdots x_\ell x_0\) and let \(G' = G - V(C)\).
	Suppose, without loss of generality, that \(u = x_0\),  and let \(v = x_k\).
	Note that \(C\) and \(G'\) form a partition of \(G\) into two proper induced subgraphs, and hence
	\begin{equation*}
	\alpha(G) \leq \alpha(C) + \alpha(G').
	\end{equation*}

	Suppose that there exists a maximum stable set \(S\) in \(C\) such that $S \cap \{x_0, x_k\} = \emptyset$.
	Let $S'$ be a maximum stable set in $G'$, and note that $S
	\cup S'$ is a stable set in \(G\).
	Thus \(\alpha(G) \geq |S \cup S'| = \alpha(C) + \alpha(G') \geq \alpha(G)\), and hence the result holds with \(H_1 = C\) and \(H_2 = G' = G - V(C)\).
	Therefore we may assume that every maximum stable set in \(C\) contains \(x_0\) or \(x_k\) (possibly both).
	Now, suppose that there exists a maximum stable set \(S\) in \(C\) such that \(\{x_0, x_k\} \subseteq S\).
	Note that if \(x_i \in S\), then \(x_{i + 1 \pmod{\ell + 1}} \notin S\), and hence \(\{x_{i + 1 \pmod{\ell + 1}} \colon x_i \in S\}\) is a maximum stable set in \(C\) containing neither \(x_0\) nor \(x_k\), a
	contradiction.
	Therefore, we may assume that for every maximum stable set \(S\) in \(C\), it follows that \(|S \cap \{x_0, x_k\}| = 1\).

	Now, we claim that there is no path of even length between $x_0$ and $x_k$ in $C$.
	Towards a contradiction, suppose the opposite and let \(P \subset C\) be such path.
	Suppose, without loss of generality, that \(P = x_0 x_1 \cdots x_k\).
	Since \(P\) has even length, \(k\) is even.
	If \(x_0\) and \(x_k\) are nonadjacent, then \(\{x_0, x_2, x_4, \ldots \}  \setminus  \{x_\ell\}\) is a
	maximum stable set in \(C\) containing both \(x_0\) and \(x_k\), a contradiction.
	Otherwise, \(k = \ell\) and \(C\) is an odd cycle, and hence  \(\{x_1, x_3, x_5, \ldots\}\) is a maximum stable set in \(C\)
	containing neither \(x_0\) nor \(x_k\), a contradiction.
	Therefore, we may assume that there is no path in \(C\) of even length between $x_0$ and
	$x_k$, and hence $C$ is an even cycle and \(k\) is odd.
	
	Suppose that there exists a maximum stable set $S'$ in $G'$
	such that $S' \cap N_{G}(x_0) \cap V(G') = \emptyset$.
	Let \(S = \{x_0, x_2, x_4, \ldots\}\), and note that \(S\) is a stable set in \(C\) such that \(x_k \notin S\) and \(|S| = \alpha(C)\).
	Moreover, \(S \cup S'\) is a stable set in \(G\), and hence \(\alpha(G) \geq |S \cup S'| = \alpha(C) + \alpha(G') \geq \alpha(G)\), and the result holds with \(H_1 = C\) and \(H_2 = G'\).
	Thus we may assume that for every maximum stable set \(S'\) of \(G'\), we have $S' \cap N_G(x_0) \cap V(G') \neq \emptyset$ and, by symmetry, \(S' \cap N_G(x_k) \cap V(G') \neq \emptyset\).
	
	Let \(P_1 = x_1x_2 \cdots x_{k - 1}\) and let \(P_2 = x_{k + 1}x_{k + 2} \cdots x_\ell\).
	Since \(C\) is a cycle, we have \(P_1\) or \(P_2\), say \(P_1\), must be different from the null graph, and since \(C\) contains no path of even
	length between \(x_0\) and \(x_k\), we have both \(P_1\) and \(P_2\)
	are paths with even order.
	Let \(G'' = G - V(P_1)\), and note that \(P_1\) and \(G''\) split \(G\) into  two proper induced subgraphs, and hence 
	\begin{equation*}
	\alpha(G) \leq \alpha(P_1) + \alpha(G'').
	\end{equation*}
	
	Suppose that there exists a maximum stable set \(S''\) in \(G''\) such that \(|S'' \cap \{x_0, x_k\}| \leq 1\).
	Thus, \(x_0\) or \(x_k\) (or both) are not contained in \(S''\).
	Suppose without loss of generality that \(x_0 \notin S''\).
	Let \(S = \{x_1, x_3, x_5, \ldots, x_{k - 2}\}\), and note that \(S\) is a stable set of \(P_1\) such that \(|S| = \alpha(P_1)\).
	Moreover, \(S'' \cup S\) is a stable set of \(G\), and hence  \(\alpha(G) \geq |S'' \cup S| = \alpha(G'') + \alpha(P_1) \geq \alpha(G)\), and the result holds with \(H_1 = P_1\) and \(H_2 = G'' = G - V(P_1)\).
	Therefore, we may assume that for every maximum stable set \(S''\) in \(G''\), we have \(\{x_0, x_k\} \subseteq S''\).
	Thus \(x_0\) and \(x_k\) are non-adjacent, and hence \(P_2\) is not the null graph.
	Let \(S_2 = \{x_{k + 2}, x_{k + 4}, \ldots, x_{\ell - 2}\}\), and note that \(S_2\) is a stable set of \(P_2\) such that \(|S_2| = \alpha(P_2) - 1\).
	Since \(\{x_0, x_k\} \subseteq S''\), we have \(|S'' \cap V(P_2)| \leq \alpha(P_2) - 1\),  and since \(S''\) is a maximum stable set of \(G''\), we have \(|S'' \cap V(P_2)| = \alpha(P_2) - 1\), otherwise \((S'' \setminus V(P_2)) \cup S_2\) would be a stable set of \(G''\) bigger than \(S''\).
	We may assume, without loss of generality, that \(S'' \cap V(P_2) = S_2\), and as result \(x_\ell, x_{\ell - 1} \notin S''\).
	Let \(S_1 = \{x_1, x_3, \ldots, x_{k - 2}\}\), and note that \(|S_1| = \alpha(P_1)\).
	Therefore, \(S = (S'' \setminus \{x_0\}) \cup (\{x_\ell\} \cup S_1)\) is a stable set of \(G\), and hence \(\alpha(G) \geq |S| = \alpha(G'') + \alpha(P_1) \geq \alpha(G)\), and the result holds with \(H_1 = P_1\) and \(H_2 = G''\).
	This finishes the proof. 
\end{proof}

\section{Digraphs whose underlying graphs are series-parallel}
\label{sec:series-parallel}

In this section, we verify Conjectures~\ref{con:diperfect} and~\ref{con:bed} for digraphs whose underlying graphs are series-parallel.
Recall that a graph \(G\) is \emph{series-parallel} if it can be obtained from the null graph by applying the following operations repeatedly: (i)~adding a vertex \(v\) with degree at most one; (ii)~adding a loop; (iii)~adding a parallel edge; (iv)~subdividing an edge.

\begin{lemma}\label{lem:sp-2-connected-cycle-two-vertices-degree-gt-2}
	Let $G$ be a 2-connected simple series-parallel graph on $n \geq 3$ vertices.
	Then there exists an induced cycle $C$ of $G$ containing at most two vertices with degree greater than two.
\end{lemma}
\begin{proof}
	The proof is by induction on \(n\).
	If \(n = 3\) then \(G\) is isomorphic to \(K_3\) (complete graph with order~\(3\)), and the result holds.
	Now suppose that \(n > 3\), and let \(\theta_1 \theta_2 \cdots \theta_k\) be a sequence of operations applied to the null graph resulting in the graph $G$.
	Note that the operation $\theta_k$ (the last operation applied) cannot be the operations~(i), (ii), or (iii),
	otherwise $G$ would not be 2-connected and simple.
	Therefore, $\theta_k$ is an operation of type (iv) and some edge \(xy\) was subdivided by adding a vertex \(z\).
	If \(x = y\), then \(xz\) and \(yz\) would be  parallel edges, a contradiction.
	Thus,  \(x \neq y\).
	If $xy \in E(G)$, then $xzyx$ is an induced cycle and the result follows.
	Thus, we may assume that $xy \notin E(G)$.
	Let $G'$ be the graph obtained after applying the sequence $\theta_1 \theta_2 \cdots \theta_{k - 1}$ of operations to the null graph.
	Note that  \(G'\) is a \(2\)-connected simple series-parallel graph and that \(V(G') =  V(G) \setminus \{z\}\).
	Thus, by the induction hypothesis, \(G'\) contains an induced cycle \(C'\) containing at most two vertices with degree greater than two.
	If \(xy \notin E(C')\), let \(C = C'\), otherwise let \(C = C' - xy + xzy\).
	Thus \(C\) is an induced cycle in \(G\) containing at most two vertices with degree greater than two and the result follows.
\end{proof}

Our proofs for Conjectures~\ref{con:diperfect} and~\ref{con:bed} for digraphs whose underlying graph are series-parallel explore the fact that series-parallel graphs have a cut vertex or an induced cycle containing at most two vertices with degree greater than two.
We assume that a counterexample exists and use this fact together with the partition lemmas from Section~\ref{sec:partition-lemmas} to contradict Lemmas~\ref{lem:dip-vertex-partition-sum-alpha} or~\ref{lem:bed-vertex-partition-sum-alpha}.
In addition to results presented in Section~\ref{sec:structural-results}, we need the following two auxiliary results to deal with a degenerated case.

\begin{lemma}\label{lem:diperfect-cycle}
	If $D \in \mathfrak{B}$ is a digraph such that $U(D)$ is a cycle, then $D$ is \(\alpha\)-diperfect.
\end{lemma}
\begin{proof}
	Let \(D\) be a digraph as in the statement.
	It is easy to verify that the result holds when \(|V(D)| = 3\), so assume that \(|V(D)| > 3\).
	Let \(H\) be a proper induced subdigraph of \(D\).
	Note that each component of \(U(H)\) is a path, and hence \(U(H)\) is a bipartite graph, which is perfect.
	Since Conjecture~\ref{con:diperfect} was verified for digraphs whose underlying graphs are perfect~\cite{Berge1982b}, \(H\) satisfies the \(\alpha\)-property.
	Thus, in order to prove the result, it just remains to show that \(D\) satisfies the \(\alpha\)-property.
	If \(|V(D)|\) is even, then \(U(D)\) is bipartite, and the result follows as before.
	Thus suppose that \(|V(D)|\) is odd.
	Let \(U(D) = x_0x_1 \cdots x_{2k} x_0\), and note that \(\alpha(D) = k\).
	Let \(S\) be a maximum stable set of \(D\).
	We may assume, without loss of generality, that \(S = \{x_0, x_2, x_4\ldots, x_{2k - 2}\}\).
	Since \(x_ix_{i + 1} \in E(U(D))\), it follows that the set \(\{x_ix_{i+1}, x_{i+1}x_i\} \cap A(D)\) is not empty, and hence let \(e_{i,i+1}\) be an arc in such set.
	
	Suppose that there exists a path \(P \subset D[\{x_{2k - 2}, x_{2k - 1}, x_{2k}, x_0\}]\) of length two such that \(V(P) = \{x_{2k}, x_{2k - 1}, z\}\), where either \(z = x_0\) or \(z = x_{2k - 2}\).
	Suppose, without loss of generality, that \(z = x_0\), and hence \(\{P\} \cup \{e_{i, i+1} \: x_{i + 1} \in S\}\) is an \(S\)-path partition of \(D\) and the result holds.
	Thus, we may assume that such path does not exist, and hence \(A(D[\{x_{2k - 2}, x_{2k - 1}, x_{2k}, x_0\}])\) is either \(\{x_{2k - 2}x_{2k - 1}, x_{2k} x_{2k -1}, x_{2k} x_0\}\) or \(\{x_{2k - 1}x_{2k - 2},\) \( x_{2k -1}x_{2k} ,  x_0x_{2k}\}\).
	Suppose, without loss of generality, that the former holds.
	If there exists a path \(P = u_1x_iu_2\subset D\) such that \(x_i \in S\) and \(\{u_1, u_2\} = \{x_{i - 1}, x_{i + 1}\}\) (where the index is taken modulus \(2k + 1\)), then \(\{P\} \cup \{e_{j-1, j} \colon x_{j} \in S \tand j < i\} \cup \{e_{j, j+1} \: x_j \in S \tand j > i\}\) is an \(S\)-path partition of \(D\) and the result holds.
	Otherwise, every \(x_i \in S\) is either a source or a sink, and hence \(D\) is an anti-directed odd  cycle, a contradiction to \(D \in \fB\).
\end{proof}

\begin{lemma}\label{lem:be-diperfect-cycle}
	If $D \in \fD$ is a digraph such that $U(D)$ is a cycle, then $D$ is BE-diperfect.
\end{lemma}
\begin{proof}
	Let \(D\) be a digraph as in the statement.
	It is easy to verify that the result holds when \(|V(D)| = 3\), so assume that \(|V(D)| > 3\).
	Let \(H\) be a proper induced subdigraph of \(D\).
	Note that each component of \(U(H)\) is a path, and hence \(U(H)\) is a bipartite graph.
	Since bipartite graphs are perfect graphs, \(H\) satisfies the BE-property by Theorem~\ref{the:bed-perfect-graphs}.
	Thus, it just remains to show that \(D\) satisfies the BE-property.
	If \(|V(D)|\) is even, then \(U(D)\) is also bipartite, and the result follows as before.
	Thus suppose that \(|V(D)|\) is odd.
	Let \(U(D) = x_0x_1 \cdots x_{2k} x_0\), and note that \(\alpha(D) = k\).
	Let \(S\) be a maximum stable set of \(D\).
	We may assume, without loss of generality, that \(S = \{x_0, x_2, x_4\ldots, x_{2k - 2}\}\).
	Since \(x_ix_{i + 1} \in E(U(D))\), the set \(\{x_ix_{i+1}, x_{i+1}x_i\} \cap A(D)\) is not empty, and hence let \(e_{i,i+1}\) be an arc in such set.
	Since \(x_{2k - 1}, x_{2k} \notin S\) and  \((x_{2k - 1},x_{2k})\) is not a blocking pair, there exists a path \(P \subset D[\{x_{2k - 2}, x_{2k - 1}, x_{2k}, x_0\}]\) of length two such that \(V(P) = \{x_{2k}, x_{2k - 1}, z\}\), where either \(z = x_0\) or \(z = x_{2k - 2}\) and \(P\) starts or ends at~\(z\).
	Suppose, without loss of generality, that \(z = x_0\), and hence \(\{P\} \cup \{e_{i, i+1} \: x_{i + 1} \in S\}\) is an \SBE-path partition of \(D\) and the result follows.
\end{proof}

The next result verifies Conjecture~\ref{con:diperfect} for digraphs whose underlying graphs are series-parallel.

\begin{theorem}\label{the:diperfect-series-parallel}
	If \(D \in \fB\) is a digraph such that $U(D)$ is series-parallel, then $D$ is \(\alpha\)-diperfect.
\end{theorem}
\begin{proof}
    Towards a contradiction, suppose that the result does not hold, and let $D$ be a counterexample with the smallest number of vertices.
    It not hard to check that \(D\) has order at least~\(3\).
    Moreover, since every subgraph of a series-parallel graph is also a series-parallel graph, we have every proper induced subdigraph of \(D\) satisfies the \(\alpha\)-property.
    We claim that \(D\) can be partitioned into two proper induced subdigraphs \(D_1\) and \(D_2\) such that \(\alpha(D) = \alpha(D_1) + \alpha(D_2)\).
    If \(D\) has a cut vertex, then the claim follows by Lemma~\ref{lem:clique-cut-alpha=alpha1+alpha2}.
    Thus suppose that \(D\) has no cut vertex, and hence \(U(D)\) is a \(2\)-connected series-parallel graph, and by Lemma~\ref{lem:sp-2-connected-cycle-two-vertices-degree-gt-2}, there exists an induced cycle \(C\) in \(U(D)\) containing at most two vertices with degree greater than \(2\).
    If \(U(D)\) is isomorphic to \(C\), then by Lemma~\ref{lem:diperfect-cycle}, \(D\) is \(\alpha\)-diperfect, a contradiction.
    Thus \(C\) is a proper induced cycle of \(U(D)\), and hence the claim follows by Lemma~\ref{lem:cycle-at-most-two-vertices-degree-greater-2}.
    Therefore \(D\) satisfies the \(\alpha\)-property by Lemma~\ref{lem:dip-vertex-partition-sum-alpha}, a contradiction.
\end{proof}

Next theorem verifies Conjecture~\ref{con:bed} for digraphs whose underlying graphs are series-parallel.
We omit its proof since it is analogous to the proof of Theorem~\ref{the:diperfect-series-parallel}, with the exception that we use Lemmas~\ref{lem:bed-vertex-partition-sum-alpha}  and~\ref{lem:be-diperfect-cycle} instead of Lemmas~\ref{lem:dip-vertex-partition-sum-alpha} and~\ref{lem:diperfect-cycle}, respectively.

\begin{theorem}\label{the:be-diperfect-series-parallel}
	If \(D \in \fD\) is a digraph such that $U(D)$ is series-parallel, then $D$ is BE-diperfect.
\end{theorem}

\section{Locally in-semicomplete digraphs}
\label{sec:in-semicomplete}

In this section, we verify Conjectures~\ref{con:diperfect} and~\ref{con:bed} for in-semicomplete digraphs.
Bang-Jensen, Huang, and Prisner~\cite{BaHuPri1993} presented the following characterization for strong in-semicomplete digraphs.

\begin{theorem}[Bang-Jensen, Huang, and Prisner, 1993]\label{the:lis-hamiltonian-iff-strong}
	An in-semicomplete digraph \(D\) of order at least \(2\) is hamiltonian if, and only if, \(D\) is strong.
\end{theorem}

Note that if a digraph contains a Hamilton cycle, then it is easy to show it satisfies the BE-property.
So, as a corollary of Theorem~\ref{the:lis-hamiltonian-iff-strong}, it follows that.

\begin{corollary}\label{cor:lis-strong}
  If \(D\) is a strong in-semicomplete digraph, then \(D\) satisfies the BE-property.
\end{corollary}

Thus in order to verify Conjectures~\ref{con:diperfect} and~\ref{con:bed} for in-semicomplete digraphs, it remains to check the conjecture for non-strong in-semicomplete digraphs.
The following is a useful result for non-strong in-semicomplete digraphs presented by Bang-Jensen and Gutin~\cite{BangJensenGutin1998}.
	
\begin{lemma}[Bang-Jensen and Gutin, 1998]\label{lem:dominate-one-dominate-all-strong-component}
	Let $D$ be an in-semicomplete digraph, and let $X$ and $Y$ be two distinct strong components of $D$.
	If there exist a vertex $x \in V(X)$ and a vertex $y \in V(Y)$  such that \(xy \in A(D)\), then \(x \mapsto Y\).
\end{lemma}

The \emph{strong component digraph} of \(D\), denoted by \(SC(D)\),  is the digraph obtained by contracting strong components of \(D\) and deleting any parallel edges and loops obtained in the process.
More precisely, if \(X_1, X_2, \ldots, X_\ell\) are the strong components of \(D\), then \(V(SC(D)) = \{x_i \: i = 1, 2, \ldots, \ell\}\) and \(E(SC(D)) = \{x_ix_j \: u \in V(X_i), v \in V(X_j), \tand uv \in A(D)\}\).
A well-known property about strong component digraphs is that they are acyclic.
Next result verifies Conjecture~\ref{con:diperfect} for in-semicomplete digraphs.

\begin{theorem}\label{the:lis-diperfect}
	If \(D\) is an in-semicomplete digraph, then $D$ is \(\alpha\)-diperfect.
\end{theorem}
\begin{proof}
    Towards a contradiction, suppose that the result does not holds and let \(D\) be a counterexample with the smallest number of vertices.
    Note that if \(H\) is an induced subdigraph of~\(D\), then \(H\) is an in-semicomplete digraph.
    Thus by the minimality of \(D\), every proper induced subdigraph of \(D\) satisfies the \(\alpha\)-property.

    If \(D\) is disconnected, then let \(C\) be a component of \(D\).
    Thus \(\alpha(D) = \alpha(C) +\alpha(D - V(C))\), and hence by Lemma~\ref{lem:dip-vertex-partition-sum-alpha}, \(D\) satisfies the \(\alpha\)-property, a contradiction.
    Therefore, we may assume that \(D\) is connected.
    If \(D\) is strong, then the result follows by Corollary~\ref{cor:lis-strong}.
    Thus, suppose that \(D\) is not strong.
    Let \(x\) be a sink in \(SC(D)\) and let \(X\) be the strong component of \(D\) which yields \(x\).
    Let \(Y = \{v \in V(D) \setminus V(X) \colon vu \in A(D) \tand u \in V(X)\}\).
    Since \(D\) is a non-strong connected digraph and \(x\) is a sink of \(SC(D)\), we have \(Y \neq \emptyset\).
    Now we claim that \(D[Y]\) is a semicomplete digraph.
    Let \(u, v \in Y\) and let \(z \in V(X)\).
    By Lemma~\ref{lem:dominate-one-dominate-all-strong-component}, \(u \mapsto z\) and \(v \mapsto z\), and since \(D\) is an in-semicomplete digraph, the vertices \(u\) and \(v\) are adjacent.
    Since \(u, v \in Y\) were chosen in an arbitrary way, we have \(D[Y]\) is a semicomplete digraph.

    Now suppose that \(Y\) is a cut of \(D\).
    Thus, by Lemma~\ref{lem:clique-cut-alpha=alpha1+alpha2}, \(D\) can be partitioned into two proper induced subdigraphs \(X_1\) and \(X_2\) such that \(\alpha(D) = \alpha(X_1) + \alpha(X_2)\), and hence by Lemma~\ref{lem:dip-vertex-partition-sum-alpha}, \(D\) satisfies the \(\alpha\)-property, a contradiction.
    Hence, we may assume that \(Y\) is not a vertex cut,  and thus \(V(D) = Y \cup V(X)\).
    Let \(u\) be a vertex in \(Y\).
    By construction \(uv \in A(D)\) for some vertex \(v \in V(X)\), and hence by Lemma~\ref{lem:dominate-one-dominate-all-strong-component}, \(u \mapsto V(X)\).
    Moreover, since \(D[Y]\) is a semicomplete digraph, we have \(u\) is an universal vertex, and hence by the minimality of \(D\) and by Lemma~\ref{lem:universal-diperfect}, \(D\) satisfies the \(\alpha\)-property, a contradiction.
\end{proof}

The following result verifies Conjecture~\ref{con:bed} for in-semicomplete digraphs.
We omit its proof since it is essentially the same proof presented in Theorem~\ref{the:lis-diperfect}, with the exception that we use Lemmas~\ref{lem:universal-be-diperfect} and~\ref{lem:bed-vertex-partition-sum-alpha} instead of  Lemmas~\ref{lem:universal-diperfect} and~\ref{lem:dip-vertex-partition-sum-alpha}.

\begin{theorem}\label{the:lis-be-diperfect}
	If \(D \in \fD\) is an in-semicomplete digraph, then $D$ is BE-diperfect.
\end{theorem}

\section{\(k\)-semi-symmetric digraphs}
\label{sec:k-semi-symmetric}

We say that an arc \(uv\) is \emph{lonely} if \(u \mapsto v\).
A digraph \(D\) is \emph{\(k\)-semi-symmetric} if it contains at most \(k\) lonely edges.
In particular, we say that a digraph is \emph{symmetric} if it is \(0\)-semi-symmetric.
Berge~\cite{Berge1982b} showed that symmetric digraphs are \(\alpha\)-diperfect, and hence Conjecture~\ref{con:diperfect} holds for them.
In this section, we use the same idea as those used by Berge to show that \(2\)-semi-symmetric digraphs and a special case of \(3\)-semi-symmetric digraph are BE-diperfect, which confirm Conjectures~\ref{con:diperfect} and~\ref{con:bed} for these classes of digraphs.

We start by describing  Berge's proof for symmetric digraphs.
Let \(D\) be a symmetric digraph.
To show that \(D\) is \(\alpha\)-diperfect, it is sufficient to show that \(D\) satisfies the \(\alpha\)-property.
Let \(S\) be a maximum stable set of \(D\) and let \(D'\) be the digraph obtained from \(D\) by removing every arc entering \(S\).
Note that, since \(D\) is symmetric, \(\alpha(D') = \alpha(D)\) and that every vertex in \(S\) is a source.
Now let \(\sP\) be a path partition of \(D'\) with the smallest size.
Since every vertex \(x \in S\) is a source, we have \(x\) is the initial vertex of some path in \(\sP\), and hence \(|S| \leq |\sP| = \pi(D')\).
By Theorem~\ref{the:gallai-milgram}, we have \(\pi(D') \leq \alpha(D') = |S|\).
Thus \(\sP\) is an \(S\)-path partition of \(D'\) and, consequently, of \(D\).
Since~\(S\) was chosen arbitrarily, we have \(D\) satisfies the \(\alpha\)-property.
Indeed,~\(\sP\) is an \SBE-path partition of \(D\), so Berge's proof shows that symmetric digraphs are BE-diperfect.
We use this idea to state the following two auxiliary results which generalize the aforementioned idea.

\begin{lemma}\label{lem:same-direction-lonely-edges}
	Let $D$ be a digraph  and let \(S\) be a maximum stable set of \(D\).
	If all the lonely edges of \(D\) that have a vertex in \(S\) are leaving \(S\), then there exists an \(S\)-path partition of \(D\) such that very path in \(\sP\) starts in \(S\).
	Analogously, if all the lonely edges of \(D\) that have a vertex in \(S\) are entering \(S\), then there exists an \(S\)-path partition \(\sP\) of \(D\) such that every path in \(\sP\) ends in \(S\).
\end{lemma}
\begin{proof}
    Analogous to Berge's proof for symmetric digraphs.
\end{proof}

\begin{lemma}\label{lem:lonely-edge-only-one-different}
	Let \(D\) be a digraph, \(S\) be a maximum stable set of \(D\), \(X_1\) be the set of lonely edges entering \(S\), and \(X_2\) be the set of lonely edges leaving \(S\).
	If all lonely edges are in \(X_1 \cup X_2\), then there exists an \SBE-path partition of \(D\).
\end{lemma}

\begin{proof}
	Let \(X'_1 = \{yx \: xy \in X_1\}\), and let \(D'\) be the digraph obtained from \(D\) by adding the edges from \(X'_1\) and by removing all the edges entering \(S\).
	Note that \(\alpha(D') = \alpha(D)\) and that all the lonely edges of \(D'\) that have an endvertex in \(S\) are leaving \(S\).
    By Lemma~\ref{lem:same-direction-lonely-edges}, there exists a path partition \(\sP'\) of \(D'\) such that every path in \(\sP'\) starts in \(S\).
	Now we show how to construct an \SBE-path partition \(\sP\) of \(D\) from the path partition \(\sP'\).
	Let \(P' = u_1u_2 \cdots u_\ell \in \sP'\).
	If \(P'\) contains no arc from \(X'_1\), then \(P'\) is a path of \(D\), and hence add \(P'\) to \(\sP\).
	Thus, suppose that \(P'\) contains an arc \(yx \in X'_1\).
    Since every path in \(\sP'\) starts in \(S\), we have \(u_1\) is the only vertex from \(P'\) in \(S\), and hence \(u_1 = y\), \(u_2 = x\), \(u_i \leftrightarrow_{D} u_{i + 1}\) for every \(i \geq 2\).
	Thus \(u_\ell u_{\ell - 1} \cdots u_1\) is a path in \(D\) and we add it to \(\sP\).
	Therefore, \(\sP\) is an \SBE-path partition of \(D\), and the result follows.
\end{proof}

Note that every \(2\)-semi-symmetric digraph belongs to \(\fD \subset \fB\).
The next result verifies Conjecture~\ref{con:diperfect} for \(2\)-semi-symmetric digraphs, and consequently for symmetric digraphs.

\begin{theorem}\label{lem:2-semi-symmetric-be-diperfect}
	If \(D\) is a \(2\)-semi-symmetric digraph, then \(D\) is BE-diperfect.
\end{theorem}
\begin{proof}
	Since every induced subdigraph of \(D\) is also \(2\)-semi-symmetric, it suffices to show that \(D\) satisfies the BE-property.
	Let \(S\) be a maximum stable set in \(D\).
	If \(D\) has precisely two lonely edges \(uv\) and \(xy\) such that \(u, y \in S\), then the result follows by Lemma~\ref{lem:lonely-edge-only-one-different}.
	Now, we can suppose that all lonely edges of \(D\) that have a vertex in \(S\) are either entering or leaving \(S\), and hence the result follows by Lemma~\ref{lem:same-direction-lonely-edges}.
\end{proof}

\begin{corollary}\label{cor:uni-leq-2-be-diperfect}
	If \(D\) is a \(2\)-semi-symmetric digraph, then \(D\) is \(\alpha\)-diperfect.
\end{corollary}

The proof of Theorem~\ref{lem:2-semi-symmetric-be-diperfect} is essentially Berge's proof for symmetric digraphs.
Our aim in proposing the class of \(k\)-semi-symmetric digraphs is extend Berge's result for \(k \geq 3\).
Next theorem is a first step in this direction.

\begin{theorem}\label{the:bed-3-semi-symmetric}
	Let \(D\) be a \(3\)-semi-symmetric digraph.
	If no pair of lonely edges has a common endvertex, then \(D\) is BE-diperfect.
\end{theorem}
\begin{proof}
	Towards a contradiction, suppose the opposite, and let \(D\) be a counterexample that minimizes \(|V(D)| + |A(D)|\).
	By the minimality of \(D\), every proper induced subdigraph of \(D\) satisfies the BE-property.
	Thus since \(D\) is a counterexample, we have \(D\) does not satisfy the BE-property.
	Let \(S\) be a maximum stable set of \(D\) for which there is no \SBE-path partition of \(D\).
	By Lemma~\ref{lem:same-direction-lonely-edges}, we may assume that there is a lonely arc \(x_1x_2\) leaving \(S\) and a lonely arc \(y_1y_2\) entering \(S\), and by Lemma~\ref{lem:lonely-edge-only-one-different}, we may assume that there is a lonely arc \(w_1w_2\) such that \(w_1, w_2 \notin S\).

	First, suppose that \(y_1\) is the only neighbor of \(y_2\).
	Let \(D' = D - \{y_1, y_2\}\) and \(S' = S \setminus \{y_2\}\).
	Note that \(\alpha(D') = \alpha(D) - 1\) and  \(S'\) is a maximum stable set of \(D'\).
	By the minimality of \(D\), there exists an \(S'_\text{BE}\)-path partition \(\sP'\) of \(D'\), and hence \(\sP' \cup \{y_1y_2\}\) is an \SBE-path partition of \(D\), a contradiction.
	Therefore, we may assume that \(y_2\) has a neighbor \(z\) different from \(y_1\).
	Since no pair of lonely edges shares a vertex, we have \(z \leftrightarrow y_2\). 
	Let \(D' = D - \{y_2z, zy_2\}\), and note that \(D'\) satisfies all the lemma's conditions.
	If \(\alpha(D') = \alpha(D)\), then \(S\) is a maximum stable set of \(D'\), and hence by the minimality of \(D\), there exists an \SBE-path partition \(\sP'\) of \(D'\).
	By the construction of \(D'\), we have \(P'\) is also an \SBE-path partition of \(D\), a contradiction.
	Therefore, we may assume that \(\alpha(D') = \alpha(D) + 1\), and let \(S'\) be a maximum stable set of \(D'\).
	Note that \(y_2, z \in S'\) and \(y_1 \notin S'\).
	Let \(S'' = S' \setminus \{y_2\}\), and note that \(S''\) is a maximum stable set of \(D\) such that \(z \in S''\) and \(y_1 \notin S''\).

	Let \(R = S \setminus S''\) and let \(Z = S'' \setminus  S\).
	Note that \(y_2 \in R\) and that \(|R| = |Z|\).
	Now we claim that there exists a perfect matching between the vertices of \(R\) and \(Z\).
	Towards a contradiction, suppose that such matching does not exist.
	By Hall's Theorem, there exists a set \(R' \subseteq R\) such that
	the neighborhood \(Z'\) of \(R'\) in \(Z\) (i.e., \(N(R') \cap Z\)) is smaller than \(|R'|\).
	Note that every vertex in \(R'\) is non-adjacent to every vertex in \(Z \setminus Z'\).
	Since \(R' \subseteq S\) is a stable set, we have \(R' \cup (S'' \setminus Z')\) is a stable set of \(D\) with size greater than \(S\), a contradiction.
	Therefore, let \(M\) be a perfect matching between the vertices of \(R\) and \(Z\).
	Given a vertex \(x \in R\), we denote by \(M(x)\) the vertex of \(Z\) matched to \(x\) by \(M\).
	
	Now suppose that for every vertex \(x \in R\), we have \(x \leftrightarrow M(x)\).
	Let \(D'' = D - R\) and note that \(S''\) is a maximum stable set of \(D''\).
	By the minimality of \(D\), we have \(D''\) satisfies the BE-property, and hence let \(\sP''\) be an \(S''_\text{BE}\)-path partition of \(D''\).
	Let \(\sP''_Z \subseteq \sP''\) be the set of paths starting or ending at \(Z\), and let \(\sP''_{\bar{Z}} = \sP'' \setminus \sP''_Z\).
	Note that \(\sP''_{\bar{Z}}\) is the set of paths starting or ending at \(S'' \setminus Z = S \setminus R\).
	For every \(x \in R\), let \(P''_x\) be the path in \(\sP''_Z\) starting or ending in \(M(x)\).
	Since \(x \leftrightarrow M(x)\), either \(xP''_x\) or \(P''_x x\) is a path of \(D\), so name  such path  \(P_x\).
	Thus, \(\{P_x \: x \in R\} \cup \sP''_{\bar{Z}}\) is an \SBE-path partition of \(D\), a contradiction.

	Thus we may assume that there exists a vertex \(x \in R\) such that \(x\) and \(M(x)\) are joined by a lonely arc.
	Since \(y_1 \notin S''\), and \(x_1x_2, y_1y_2\) are the only lonely edges with an endvertex in \(S\), we have \(x = x_1\) and \(M(x) = x_2\).
	Moreover, for each vertex \(y \in R \setminus \{x_1\} \), we have \(y \leftrightarrow M(y)\).
	Let \(D'' = D - R - \{yx_2 \: yx_2 \in A(D)\}\).
	Note that we only remove non-lonely edges, since the arc \(x_1x_2\) is a lonely arc  incident to \(x_2\) and no pair of lonely edges share a vertex. 
	Thus \(x_2\) is a source in \(D''\) and \(S''\) is a maximum stable set of \(D''\).
	Since \(D - R\) has only one lonely arc, either all the lonely edges of \(D''\) have an endvertex in \(S''\) or every lonely arc that has an endvertex in \(S''\) are leaving  \(S''\).
	Hence by Lemma~\ref{lem:same-direction-lonely-edges} or~\ref{lem:lonely-edge-only-one-different}, there exists an \(S''_\text{BE}\)-path partition \(\sP''\) of \(D''\).
	Let \(\sP''_Z \subseteq \sP''\) be the set of paths starting or ending at \(Z\), and let \(\sP''_{\bar{Z}} = \sP'' \setminus \sP''_Z\).
	Note that \(\sP''_{\bar{Z}}\) is the set of paths starting or ending at \(S'' \setminus Z = S \setminus R\).
	For every \(y \in R\), let \(P''_y\) be the path in \(\sP''_Z\) starting or ending at \(M(y)\).
	Since \(M(x) = x_2\) is a source in \(D''\), \(P''_x\) is a path starting at \(M(x)\), and hence \(P_{x} = xP''_x\) is a path of \(D\).
	Since \(y \leftrightarrow M(y)\), for every \(y \in R \setminus \{x\}\), we have either \(yP''_y\) or \(P''_y y\) is a path of \(D\), so name  such path \(P_y\).
	Thus, \(\{P_y \: y \in R\} \cup \sP''_{\bar{Z}}\) is an \SBE-path partition of \(D\), a contradiction.
\end{proof}

\begin{corollary}
	Let \(D\) be a \(3\)-semi-symmetric digraph.
	If no pair of lonely edges has a common endvertex, then \(D\) is \(\alpha\)-diperfect.
\end{corollary}

\section{Concluding Remarks}
\label{sec:concluding-remarks}

Berge's Conjecture for \(\alpha\)-diperfect digraphs has proved through the years to be a hard to solve problem.
The only previous known results for this conjecture were those presented by Berge in his seminal paper from 1982.
Aiming to understand this difficulty, we proposed the study of a new class of digraphs, that we called BE-diperfect, and presented Conjecture~\ref{con:bed} which would characterize such class. 
We refer to this conjecture as \emph{Begin-End conjecture}.
We believe that the study of BE-diperfect digraphs can lead to further results for Berge's Conjecture.
Indeed, our results for Berge's Conjecture are a development of previous results we obtained for the Begin-End conjecture.

Our results for digraphs whose underlying graph are series-parallel and for in-semicomplete digraphs rely on the structure of the underlying graph.
With the help of lemmas from Section~\ref{sec:partition-lemmas}, we were able to split a digraph \(D\) in any of those families into two proper induced subdigraphs \(D_1\) and \(D_2\) with \(\alpha(D) = \alpha(D_1) + \alpha(D_2)\), and apply induction. 
We believe that this approach, with the development of new partition lemmas, could yield new results for others classes of digraphs.
On the other hand, our result for \(3\)-semi-symmetric digraphs rely more on the orientation of the arcs.
The study of these conjectures for \(k\)-semi-symmetric digraphs with \(k \geq 3\) seems an inviting direction of research since it allows an incremental approach for dealing with the lonely arcs and the forbidden subdigraphs.

\bibliographystyle{amsplain}

\end{document}